\documentclass{amsart}
\usepackage[english]{babel}
\usepackage{amssymb,stmaryrd,enumerate,mathrsfs,bbm}
\usepackage{amsthm}

\catcode`\<=\active \def<{
\fontencoding{T1}\selectfont\symbol{60}\fontencoding{\encodingdefault}}
\catcode`\>=\active \def>{
\fontencoding{T1}\selectfont\symbol{62}\fontencoding{\encodingdefault}}
\newcommand{\assign}{:=}
\newcommand{\cdummy}{\cdot}
\newcommand{\mathD}{\mathrm{D}}
\newcommand{\mathd}{\mathrm{d}}
\newcommand{\nocomma}{}
\newcommand{\nospace}{}
\newcommand{\nosymbol}{}
\newcommand{\precprec}{\prec\!\!\!\prec}
\newcommand{\tmdummy}{$\mbox{}$}
\newcommand{\tmop}[1]{\ensuremath{\operatorname{#1}}}
\newcommand{\tmscript}[1]{\text{\scriptsize{$#1$}}}
\newcommand{\tmtextit}[1]{{\itshape{#1}}}
\newenvironment{enumeratenumeric}{\begin{enumerate}[1.] }{\end{enumerate}}
\newenvironment{itemizedot}{\begin{itemize} }{\end{itemize}}

\newtheorem{theorem}{Theorem}

\newtheorem{lemma}[theorem]{Lemma}

\theoremstyle{remark}
\newtheorem{remark}[theorem]{Remark}
\theoremstyle{definition}
\newtheorem{definition}[theorem]{Definition}

\newcommand{\CC}{\mathscr{C} \hspace{.1em}}
\newcommand{\CD}{\mathscr{D} \hspace{.1em}}
\newcommand{\CF}{\mathscr{F} \hspace{.1em}}
\newcommand{\CB}{\mathscr{B} \hspace{.1em}}
\newcommand{\CA}{\mathscr{A} \hspace{.1em}}
\newcommand{\DD}{\mathscr{D} \hspace{.1em}}
\newcommand{\LL}{\mathscr{L} \hspace{.2em}}

\newcommand{\PP}{\mathscr{P} \hspace{.1em}}
\newcommand{\DF}{\mathfrak{D}}

\begin{document}

\title{Paracontrolled quasilinear SPDEs}

\author{M.~Furlan}
\address{CEREMADE\\
PSL-Universit{\'e} Paris Dauphine, France}
\email{furlan@ceremade.dauphine.fr}

\author{ M.~Gubinelli}
\address{IAM \& HCM\\
Universit{\"a}t Bonn, Germany}
\email{gubinelli@iam.uni-bonn.de}

\begin{abstract}
  We introduce a non-linear paracontrolled calculus and use it to renormalise
  a class of singular SPDEs including certain quasilinear variants of the
  periodic two dimensional parabolic Anderson model.
\end{abstract}

{\maketitle}

\section{Introduction}

We show how to renormalise a class of general quasilinear equations of which
one of the simplest examples is the following parabolic SPDE:
\begin{equation}
  \partial_t u (t, x) - a (u (t, x)) \Delta u (t, x) = \xi (x), \qquad u (0,
  x) = u_0 (x), \quad x \in \mathbbm{T}^2, t \geqslant 0,
  \label{eq:ur-formulation}
\end{equation}
with $a : \mathbbm{R} \rightarrow [\lambda, 1]$ for~$\lambda > 0$ a uniformly
bounded $C^3$ diffusion coefficient, and $\| a^{(k)} \|_{L^{\infty}} \leqslant
1$ for $k = 0, \ldots, 3$. We assume that $\xi \in \CC \nospace^{\alpha - 2}
(\mathbbm{T}^2)$ with $2 / 3 < \alpha < 1$ where $\CC^{\alpha}
(\mathbbm{T}^2)$ is the Besov space $B^{\alpha}_{\infty, \infty}
(\mathbbm{T}^2)$. This would apply for example to the space white noise on
$\mathbbm{T}^2$. In this case we only expect that $u (t, \cdot) \in
\CC^{\alpha} (\mathbbm{T}^2)$ and the term $a (u (t, \cdot)) \Delta u (t,
\cdot)$ is not well defined when $2 \alpha - 2 < 0$.
Eq.~(\ref{eq:ur-formulation}) is a quasilinear generalisation of the
two--dimensional periodic parabolic Anderson model (PAM).

Let us remark from the beginning that the framework developed in this paper
allows to deal with a class of equations of the form
\begin{equation}
  a_1 (u (t, x)) \partial_t u (t, x) - a_2 (u (t, x)) \Delta u (t, x) = \xi
  (a_3 (u (t, x)), t, x), \qquad x \in \mathbbm{T}^2, t \geqslant 0,
  \label{eq:general-model}
\end{equation}
where $a_1, a_2$ are sufficiently smooth non-degenerate coefficients and $\xi
(z, t, x)$ is a Gaussian process with covariance
\[ \mathbbm{E} [\xi (z, t, x) \xi (z', t', x')] = F (z, z') Q (t - t', x -
   x'), \qquad x, x' \in \mathbbm{T}^2, t, t', z, z' \in \mathbbm{R}, \]
with $F$ a smooth function and $Q$ a distribution of parabolic regularity
$\rho > - 4 / 3$. This includes as a special case the space white noise
discussed before, but we could consider a time white noise with a regular
dependence on the space variable, or some noise which is mildly irregular in
space and time.

Moreover, the scalar character of the equation or of the non-linear diffusion
coefficient will not play any specific role and we could consider
vector--valued equations with general diffusion coefficients, provided the
template problem~(\ref{eq:template}) below is uniformly parabolic.

For the sake of clarity and simplicity we will discuss mainly the basic
example~(\ref{eq:ur-formulation}) since this contains already most of the
technical difficulties. The fact that one can handle models as general
as~(\ref{eq:general-model}) can be seen as a direct byproduct of the
techniques we will introduce below. Let us state one simple result that can be
obtained via the theory developed in this paper:

\begin{theorem}
  Fix $2 / 3 < \alpha < 1$. Let $\xi \in \CC^{\alpha - 2} (\mathbbm{T}^2)$ be
  a space white noise with zero average on the torus, $u_0 \in \CC^{\alpha}$
  an initial condition and $a : \mathbbm{R} \rightarrow [\lambda, 1]$ for some
  $\lambda > 0$, $a \in C^3 (\mathbbm{R})$ and $\| a^{(k)} \|_{L^{\infty}}
  \leqslant 1$ $\forall k \in 0, \ldots, 3$. Let $(\xi_{\varepsilon}, u_{0,
  \varepsilon})_{\varepsilon > 0}$ be a family of smooth approximations of
  $\xi, u_0$ obtained by convolution with a rescaled smoothing kernel and
  $u_{\varepsilon}$ the classical solution to the Cauchy problem
  \begin{equation}
    \partial_t u_{\varepsilon} - a (u_{\varepsilon}) \Delta u_{\varepsilon} =
    \xi_{\varepsilon} + \sigma_{\varepsilon} \frac{a' (u_{\varepsilon})}{a
    (u_{\varepsilon})^2}, \qquad u (0) = u_{0, \varepsilon} . \qquad
    \label{eq:ur-formulation-approx}
  \end{equation}
  Then we can choose a sequence of constants
  $(\sigma_{\varepsilon})_{\varepsilon > 0}$ and a random time $T > 0$ in such
  a way that the family of r.v. $(u_{\varepsilon})_{\varepsilon > 0} \subseteq
  \LL_T^{\alpha} (\mathbbm{T}^2)$ converges almost surely as $\varepsilon
  \rightarrow 0$ to a random element $u \in \LL_T^{\alpha} (\mathbbm{T}^2)$,
  where $\LL_T^{\alpha}$ is the parabolic space $C \left( [0, T], \CC^{\alpha}
  (\mathbbm{T}^2) \right) \cap C^{\alpha / 2} \left( [0, T], \CC^0
  (\mathbbm{T}^2) \right)$.
  
  This element can be characterised as the solution to a paracontrolled
  singular SPDE (see below for more details).
\end{theorem}

In order to devise a suitable formulation of eq.~(\ref{eq:ur-formulation}) and
obtain a theory with $u \in \CC^{\alpha}$ we decompose the non-linear
diffusion term in the l.h.s. with the help of Bony's
paraproduct~{\cite{meyer_remarques_1981}} and write
\begin{equation}
  \partial_t u - a (u) \prec \Delta u = \xi + \Phi (u)
  \label{e:paracontrolled-eq}
\end{equation}
with
\begin{equation}
  \Phi (u) \assign a (u) \circ \Delta u + a (u) \succ \Delta u
  \label{e:def-Phi}
\end{equation}
where $\prec, \succ$ are standard paraproducts and $\circ$ denotes the
resonant product (see Appendix~\ref{appendix:besov} for the definitions).
In~(\ref{e:paracontrolled-eq}) the l.h.s. is always well defined,
irrespectively of the regularity of the function $u$, and the problem becomes
that of controlling the resonant product $a (u) \circ \Delta u$ appearing in
the r.h.s.~. A key point of the analysis put forward below is that this term
can be expected to be of regularity $2 \alpha - 2 > \alpha - 2$, better than
the leading term $\xi \in \CC^{\alpha - 2}$.

Our approach can be described as follows. For an equation of the form
\[ \partial_t u - a_1 (u) \Delta u = a_2 (u) \xi, \]
we consider at first a \tmtextit{parametric} ``template'' problem with
constant coefficients $\eta_1, \eta_2$
\begin{equation}
  \partial_t \vartheta (\eta, t, x) - \eta_1 \Delta \vartheta (\eta, t, x) =
  \eta_2 \xi (t, x) \label{eq:template}
\end{equation}
where now $\eta = (\eta_1, \eta_2)$ is a fixed vector in $\mathbbm{R}^2$. A
\tmtextit{nonlinear} paraproduct $\Pi_{\precprec}$ will allow us to modulate
the parametric solution $\vartheta (\eta)$ with the quantity $a (u) = (a_1
(u), a_2 (u))$ in order to capture the most irregular part of the solution $u$
itself. As a consequence, the paracontrolled Ansatz
\[ u = \Pi_{\precprec} (a (u), \vartheta) + u^{\sharp} \]
will define a regular remainder term $u^{\sharp}$ which solves a standard PDE.
With this decomposition the resonant products appearing in
eq.~(\ref{e:paracontrolled-eq}) can be estimated along the lines of the
standard paracontrolled arguments introduced
in~{\cite{gubinelli_paracontrolled_2015}}, and all the arguments introduced
there can be extended in a straightforward manner to the quasilinear setting.

\

Recently Otto and Weber~{\cite{otto_quasilinear_2016}} and Bailleul,
Debussche and Hofmanov{\'a}~{\cite{hofmanova}} investigated quasilinear SPDEs
in the context of pathwise methods and in a range of regularities compatible
with the ones we will consider in this paper.
\begin{itemizedot}
  \item In~{\cite{otto_quasilinear_2016}} the authors obtained a local
  well-posedness result for equations of the form
  \[ \partial_t u (t, x) - a (t, x) \partial_x^2 u (t, x) = f (u (t, x)) \xi
     (t, x), \qquad (t, x) \in \mathbbm{T}^2 \]
  where both space and time variables take values in a one dimensional
  periodic domain and their noise can be white in time but colored in space,
  essentially behaving like a distribution of parabolic regularity in $(- 4 /
  3, 1)$. In order to do that they introduce a specific notion of
  \tmtextit{modelled} function and related estimates. \ \ \ \
  
  \item Bailleul, Debussche and Hofmanov{\'a} in~{\cite{hofmanova}} obtain
  local well-posedness for the generalised parabolic Anderson model equation
  \begin{equation}
    \partial_t u (t, x) - a (u (t, x)) \Delta u (t, x) = g (u (t, x)) \xi (x)
    \label{eq:deb-hof} \qquad t \geqslant 0, x \in \mathbbm{T}^2 .
  \end{equation}
  The authors obtain the same result as the one presented in
  Section~\ref{sec:gen-nonlinear} of our work, without the machinery of
  nonlinear paraproducts introduced here, but using only the basic tools of
  paracontrolled analysis and some clever transformations. 
\end{itemizedot}
Let us comment a bit on relations of our results with those contained in these
two papers.
\begin{itemizedot}
  \item The \tmtextit{parametric} controlled Ansatz of Otto and
  Weber~{\cite{otto_quasilinear_2016}} is the main source of inspiration for
  the present work. We think that it is a very deep and fundamental
  observation which is quite orthogonal to the development of an alternative
  theory for singular SPDEs which is the main aim of their work. Our paper
  shows that this idea survives outside their specific methodology, in
  particular that Safanov's approach to Schauder estimates is not necessary to
  handle quasilinear equations, and a relatively straightfoward extension of
  the paracontrolled approach is sufficient to encompass in a compact manner
  the results on quasilinear equations contained in their paper. \
  
  \item While the approach of Bailleul, Debussche and
  Hofmanov{\'a}~{\cite{hofmanova}} is more straightforward than ours, it has
  fundamental limitations. In particular we remark that the generalization
  of~(\ref{eq:deb-hof}) to matrix-valued diffusion coefficients $(a_{i j})_{i,
  j}$, namely to equations of the form
  \begin{equation}
    \partial_t u (t, x) - a_{i j} (u (t, x)) \frac{\partial^2}{\partial x_i
    \partial x_j} u (t, x) = g (u (t, x)) \xi, \qquad t \geqslant 0, x \in
    \mathbbm{T}^2, \label{e:matrix-valued-coeff}
  \end{equation}
  is out of reach of the techniques used in~{\cite{hofmanova}}, while can be
  treated directly in our framework and by Otto and Weber's approach.
  
  \item Another interesting observation we make in the present work is that
  the parametric Ansatz can pervade without problems all the coefficients of
  the equation. In particular we can allow very general noise terms whose law
  itself depends on the solutions of the equation. This is one of the main
  novelties of this paper, not present in other works on singular SPDEs. See
  eq.~(\ref{eq:general-model}) above and Section~\ref{sec:fully-generality}
  below for a detailed explanation of this case.
  
  \item Let us mention that Otto and Weber's parametric Ansatz has
  predecessors in the theory of rough paths and in standard stochastic
  analysis. Non--linear versions of rough paths have been considered by one of
  the authors in order to study the Korteweg--de Vries
  equation~{\cite{gubinelli_rough_2012}}. Non--linear Young integrals were
  used by one of the authors in joint work with
  Catellier~{\cite{catellier_averaging_2016}} to study the the regularising
  properties of sample paths of stochastic processes. See also the related
  work of Hu and Le~{\cite{hu_nonlinear_2014}} on differential systems in
  H{\"o}lder media. Relevant to this discussion of non-linear variants of
  rough paths is the work of Bailleul on rough
  flows~{\cite{bailleul_flows_2015}} and their application to
  homogeneisation~{\cite{bailleul_rough_2016}}. By looking at the composition
  $f (g (x))$ as the action of the distribution $\delta_{g (x)}$ on the
  function $f$, non--linear constructions can be linearised at the price of
  working in infinite--dimensional spaces: this is the approach chosen by
  Kelly and Melbourne to avoid non--linear generalisations of rough path
  theory in their study of homogeneisation of fast--slow system with random
  initial conditions~{\cite{kelly_deterministic_2014}}. It is also worth
  mentioning Kunita's theory of semimartingale vector
  fields~{\cite{kunita_stochastic_1984}}, which occupies a place in stochastic
  analysis quite similar to that occupied by these recent developments in the
  landscape of rough paths/paracontrolled distributions theories. A
  non--linear generalisation of the classic bilinear paraproducts already
  appeared in the notion of \tmtextit{paracomposition} introduced by
  Alinhac~{\cite{bony_microlocal_1991}}.
\end{itemizedot}

Usefulness of paraproducts in the analysis of non-linear PDEs is by now well
established: see for example the seminal paper of
Meyer~{\cite{meyer_remarques_1981}}, the early review of
Bony~{\cite{bony_microlocal_1991}}, the recent books of Alinhac and
G{\'e}rard~{\cite{alinhac_operateurs_1991}} and Bahouri, Chemin, and
Danchin~{\cite{bahouri_fourier_2011}}. Let us mention also the interesting
paper of H{\"o}rmander~{\cite{hormander_nash_moser_1990}} where
paradifferential operators allow to bypass the Nash--Moser fixpoint theorem in
some applications \ where the loss of regularity prevents straightforward use
of standard Banach fixpoint theorem. The main observation in that paper is
that, with the aid of paradifferential operators, it is possible to identify a
``corrected'' problem for which standard Banach fixpoint applies.

Paracontrolled calculus for singular SPDEs has beed introduced by Gubinelli,
Imkeller and Perkowski~{\cite{gubinelli_paracontrolled_2015}} (see also the
lecture notes~{\cite{gubinelli_lectures_2015}}) and used to study various
equations like the KPZ equation~{\cite{gubinelli_kpz_2015}}, the dynamic
$\Phi^4_3$ model~{\cite{catellier_paracontrolled_2013}} in $d = 3$ and its
global well--posedness~{\cite{mourrat_global_2016}}, the spectrum of the
continuous Anderson Hamiltonian in $d = 2$~{\cite{allez_continuous_2015}}. By
using heat--semigroup techniques paracontrolled calculus has been extended to
the manifold context by Bailleul and Bernicot~{\cite{bailleul_heat_2015}}.

Paracontrolled calculus is currently limited to ``first order'' computations.
This limitation is also ubiquitous in the present work. Even if, in practice,
this is not a big issue, and the calculus is still able to deal with a large
class of problems, it makes the paracontrolled approach less appealing for a
general theory of singular SPDEs. Let us remark that recently Bailleul and
Bernicot~{\cite{bailleul_higher_2016}} developed an higher order version of
the paracontrolled calculus. However, apart from these recent development,
whose impact is still to be assessed, the most general theory for singular
SPDEs has been developed by
Hairer~{\cite{hairer_theory_2014,hairer_singular_2014,friz_course_2014}} under
the name of \tmtextit{regularity structures} theory. Regularity structures are
a vast generalisation of Lyons' rough
paths~{\cite{lyons_differential_1998,lyons_system_2002,lyons_differential_2007}}
which give effective tools to describe non-linear operations acting on certain
spaces of distributions, their renormalization by subtraction of local
singularities and their use to solve singular SPDEs. Regularity structures
have been successfully applied to all the models mentioned so
far~{\cite{hairer_theory_2014,hairer_solving_2013}}, to other models like the
Sine--Gordon model~{\cite{hairer_dynamical_2014}} (which however can also be
handled via paracontrolled techniques) and to study weak universality
conjectures~{\cite{hairer_class_2015,hairer_large_2016}}. In their current
instantiation it does not seem possible to solve quasilinear SPDEs via
regularity structures. The results of the present paper hint to the fact that
a non-linear version of regularity structures is conceivable, at least in
principle. Indeed one can imagine models depending on additional parameters
and modelled distributions acting as evaluations of the parametric models at
certain space--time dependent values of the parameters. It would be
interesting to pursue further this intuition.

\

The structure of the paper is the following. In Section~\ref{s:nonlinear} we
introduce our basic tools: the non-linear paraproduct decomposition and some
related commutation lemmas. In Section~\ref{s:paracontrolled} we introduce the
paracontrolled Ansatz which allows to transform the singular
problem~(\ref{eq:ur-formulation}) into a well--behaved PDE. In
Section~\ref{s:fixpoint} we discuss a-priori estimates, the uniqueness of the
solution of the transformed PDE and its continuity w.r.t. the random data and
the initial condition. We also introduce the algebraic structure which allows
to renormalise the model. Section~\ref{s:renorm} deals with the
renormalization of the stochastic data and the construction of the enhanced
noise associated to white noise. Section~\ref{sec:gen-nonlinear} deals with
the extension of the results to the more general equation~(\ref{eq:deb-hof}).
In Section~\ref{sec:fully-generality} we develop a paracontrolled structure
for equation~\ref{eq:general-model} and define its renormalised enhanced
noise. Finally, in Appendix~\ref{appendix:besov} we review some reference
material on Besov spaces and linear paraproducts, and prove some technical
lemmas.

\

\paragraph{Notations}We will denote as $\CC^{\alpha} \assign
B^{\alpha}_{\infty, \infty} (\mathbbm{T}^2)$ the Zygmund space of regularity
$\alpha \in \mathbbm{R}$ on the torus $\mathbbm{T}^2$. See
Appendix~\ref{appendix:besov} for the definition of the general Besov spaces
$B^{\alpha}_{p, q}$, the Littlewoord--Paley operators $(\Delta_i)_{i \geqslant
- 1}$ and the basic properties thereof needed in this paper. If $V$ is a
Banach space and $T > 0$, we denote $C_T^{\alpha} V$ the space of
$\alpha$--H{\"o}lder functions in $C_T V \assign C ([0, T] ; V)$. We introduce
parabolic spaces $\LL^{\alpha}_T \assign C_T^{\alpha / 2} \CC^0 \cap C_T
\CC^{\alpha}$ with norm
\begin{equation}
  \| f \|_{\LL^{\alpha}_T} = \| f \|_{C_T^{\alpha / 2} \CC^0} + \| f \|_{C_T
  \CC^{\alpha}} . \label{e:def-Lnorm}
\end{equation}
Moreover for convenience we denote $\CC_T^{\alpha} \assign C_T \CC^{\alpha}$.
We will avoid to write explicitly the time span $T$ whenever this does not
cause ambiguities. We will need also spaces for functions of $(\eta, t, x)$
where $\eta$ is an additional parameter in $[\lambda, 1]$ for $\lambda \in (0,
1)$ which we denote $C_{\eta}^k V$ with the norm
\begin{equation}
  \| F \|_{C_{\eta}^k V} = \sup_{\eta \in [\lambda, 1]} \sup_{n = 0, \ldots,
  k} \| \partial_{\eta}^n F (\eta, \cdummy) \|_V, \label{e:def-Cetanorm}
\end{equation}
where $V$ is a Banach space of functions on $[0, T] \times \mathbbm{T}^2$: in
our case $V = \CC_T^{\alpha}$ or $V = \LL^{\alpha}_T$.

We will denote by $K_{i, x} (y) = 2^{2 i} K (2^i (x - y))$ the kernel of the
Littlewood--Paley operator $\Delta_i$ and $Q_{i, t} (s) = Q_i (t - s) = 2^{2
i} Q (2^{2 i} (t - s))$ a smoothing kernel at scale $2^{2 i}$ in the time
direction where $Q$ is a smooth, positive function with compact support in
$\mathbbm{R}_+$ and mass $1$. We introduce also the shortcut $P_{i, x} = K_{<
i - 1, x} = \sum_{j < i - 1} K_{i, x}$. Another notation shortcut widely used
in this article is to write $\int_{x, y}$ for integrals on $\mathbbm{T}^2$ or
$\mathbbm{R}$ with respect to the measures $\mathd x$ and $\mathd y$ without
specifying the integration bounds, whenever this does not create ambiguity.
Finally, we will note $\delta f^{t x}_{s y} = f (t, x) - f (s, y)$ and
$\delta_{\tau} f_{s y}^{t x} = f (s, y) + \tau (f (t, x) - f (s, y))$ for
$\tau \in [0, 1]$.

\section{Nonlinear paraproducts}\label{s:nonlinear}

In this section we introduce nonlinear paraproduct and some related results
that will be used in Section~\ref{s:paracontrolled} to analyse
equation~(\ref{eq:ur-formulation}).

Let $g : [0, T] \times \mathbbm{T}^2 \rightarrow \mathbbm{R}$, and $h :
\mathbbm{R} \times [0, T] \times \mathbbm{T}^2 \rightarrow \mathbbm{R}$ be
smooth functions. We can decompose the composition $h (g (\cdummy), \cdummy)$
via \tmtextit{nonlinear} paraproducts as follows. Define
\begin{eqnarray}
  \Pi_{\prec} (g, h) (t, x) & \assign & \sum_q \int_{y, z} P_{q, x} (y) K_{q,
  x} (z) h (g (t, y), t, z)  \label{e:nonlin-par-small}\\
  \Pi_{\circ} (g, h) (t, x) & \assign & \sum_{k \sim q} \int_{y, z} K_{k, x}
  (y) K_{q, x} (z) h (g (t, y), t, z)  \label{e:nonlin-par-same}\\
  \Pi_{\succ} (g, h) (t, x) & \assign & \sum_k \int_{y, z} K_{k, x} (z) P_{k,
  x} (y) h (g (t, z), t, y) .  \label{e:nonlin-par-big}
\end{eqnarray}
This gives a map
\begin{equation}
  (g, h) \mapsto \Pi_{\nocomma \diamondsuit} (g, h) \assign \Pi_{\prec} (g, h)
  + \Pi_{\circ} (g, h) + \Pi_{\succ} (g, h) = h (g (\cdot), \cdot) 
  \label{e:nonlin-decomp}
\end{equation}
that can be uniquely extended to
\[ \Pi_{\diamondsuit} : \CC_T^{\rho} \times C^2_{\eta} \CC_T^{\gamma}
   \rightarrow \CC_T^{\gamma \wedge \rho} \qquad \rho \in (0, 1), \gamma \in
   \mathbbm{R}, \rho + \gamma > 0, \]
thanks to the following bounds:

\begin{lemma}[Nonlinear paraproduct estimates]
  \label{l:decomp-est} Let $g \in \CC_T^{\rho}$ for some $\rho \in (0, 1)$, $g
  \in [\lambda, 1]$, and $h \in C^2_{\eta} \CC_T^{\gamma}$ for any $\gamma \in
  \mathbbm{R}$. Then
  \[ \| \Pi_{\prec} (g, h) \|_{\CC^{\gamma}_T} \lesssim \| h \|_{C_{\eta}
     \CC^{\gamma}_T}, \qquad \| \Pi_{\succ} (g, h) \|_{\CC^{\rho \wedge (\rho
     + \gamma)}_T} \lesssim \| g \|_{\CC^{\rho}_T} \| h \|_{C_{\eta}^1
     \CC^{\gamma}_T}, \]
  and
  \begin{eqnarray*}
    \| \Pi_{\prec} (g_1, h) - \Pi_{\prec} (g_2, h) \|_{\CC^{\gamma}_T} &
    \lesssim & \| g_1 - g_2 \|_{C_T L^{\infty}} \| h \|_{C_{\eta}^1
    \CC^{\gamma}_T},\\
    \| \Pi_{\succ} (g_1, h) - \Pi_{\succ} (g_2, h) \|_{\CC_T^{\rho \wedge
    (\rho + \gamma)}} & \lesssim & \| g_1 - g_2 \|_{C_T L^{\infty}} \left( \|
    g_1 \|_{\CC^{\rho}_T} + \| g_2 \|_{\CC^{\rho}_T} \right) \| h
    \|_{C_{\eta}^2 \CC^{\gamma}_T}\\
    &  & + \| g_1 - g_2 \|_{\CC^{\rho}_T} \| h \|_{C_{\eta}^1 \CC^{\gamma}_T}
    .
  \end{eqnarray*}
  Moreover if $\rho + \gamma > 0$ we have also
  \begin{eqnarray*}
    \| \Pi_{\circ} (g, h) \|_{\CC^{\rho + \gamma}_T} & \lesssim & \| h
    \|_{C_{\eta}^1 \CC^{\gamma}_T} \| g \|_{\CC^{\rho}_T},\\
    \| \Pi_{\circ} (g_1, h) - \Pi_{\circ} (g_2, h) \|_{\CC_T^{\rho + \gamma}}
    & \lesssim & \| g_1 - g_2 \|_{C_T L^{\infty}} \left( \| g_1
    \|_{\CC^{\rho}_T} + \| g_2 \|_{\CC^{\rho}_T} \right) \| h \|_{C_{\eta}^2
    \CC^{\gamma}_T}\\
    &  & + \| g_1 - g_2 \|_{\CC^{\rho}_T} \| h \|_{C_{\eta}^1 \CC^{\gamma}_T}
    .
  \end{eqnarray*}
  In particular if $\rho + \gamma > 0$ the composition $\Pi_{\diamondsuit} (g,
  h) = h (g (\cdot), \cdot)$ is linear in $h$ and locally Lipshitz in $g$:
  \begin{eqnarray*}
    \| \Pi_{\diamondsuit} (g, h) \|_{\CC_T^{\gamma}} & \lesssim & \| h
    \|_{C_{\eta}^1 \CC^{\gamma}_T} \| g \|_{\CC^{\rho}_T},\\
    \| \Pi_{\diamondsuit} (g_1, h) - \Pi_{\diamondsuit} (g_2, h)
    \|_{\CC_T^{\gamma}} & \lesssim & \| g_1 - g_2 \|_{\CC^{\rho}_T} \left( 1 +
    \| g_1 \|_{\CC^{\rho}_T} + \| g_2 \|_{\CC^{\rho}_T} \right) \| h
    \|_{C_{\eta}^2 \CC^{\gamma}} .
  \end{eqnarray*}
\end{lemma}

\begin{proof}
  Using the fact that
  \begin{eqnarray*}
    \| \Delta_k h (g (t, y), t, \cdot) \|_{L^{\infty}} & \lesssim & 2^{-
    \gamma k} \| h \|_{C_{\eta} \CC_T^{\gamma}},\\
    \| \Delta_k h (g (t, y), t, \cdot) - \Delta_k h (g (t, y'), t, \cdot)
    \|_{L^{\infty}} & \lesssim & 2^{- \gamma k} \| h \|_{C_{\eta}^1
    \CC_T^{\gamma}} \| g \|_{C_T  \CC^{\rho}} | y - y' |^{\rho},
  \end{eqnarray*}
  and
  \[ \| \Delta_k h (g_1 (t, y), t, \cdot) - \Delta_k h (g_2 (t, y), t, \cdot)
     \|_{L^{\infty}} \lesssim 2^{- \gamma k} \| h \|_{C_{\eta}^1
     \CC_T^{\gamma}} \| g_1 - g_2 \|_{C_T L^{\infty}}, \]
  we obtain the bounds on $\Pi_{\prec} (g, h)$, $\Pi_{\succ} (g, h)$,
  $\Pi_{\circ} (g, h)$ and $\Pi_{\prec} (g_1, h) - \Pi_{\prec} (g_2, h)$. We
  proceed to estimate the term $\Pi_{\succ} (g_1, h) - \Pi_{\succ} (g_2, h)$.
  We will use the following notation for brevity:
  \[ \delta g_{2 z}^{1 y} : = g_1 (t, y) - g_2 (t, z) \quad \tmop{and} \quad
     \delta_{\tau} g_{2 z}^{1 y} : = g_2 (t, z) + \tau (g_1  (t, y) - g_2 (t,
     z)) . \]
  Then
  \begin{eqnarray*}
    &  & | \int_{y, z} K_{k, x} (z) P_{k, x} (y) [h (g_1 (t, z), t, y) - h
    (g_2 (t, z), t, y)] |\\
    & = & | \int_{y, z, \tau \in [0, 1]} K_{k, x} (z) P_{k, x} (y)
    [\partial_{\eta} h (\delta_{\tau} g_{2 z}^{1 z}, t, y) \delta g_{2 z}^{1
    z} - \partial_{\eta} h (\delta_{\tau} g_{2 x}^{1 x}, t, y) \delta g_{2
    x}^{1 x}] |\\
    & \lesssim & | \int_{y, z, t \in [0, 1], \sigma \in [0, 1]} K_{k, x} (z)
    P_{k, x} (y) \partial^2_{\eta} h (\delta_{\sigma} (\delta_{\tau}
    g_2^1)^z_x, t, y) (\delta g^{1 z}_{1 x} - \delta g^{2 z}_{2 x}) \delta
    g_{2 z}^{1 z} |\\
    &  & + | \int_{y, z, \tau \in [0, 1]} K_{k, x} (z) P_{k, x} (y)
    \partial_{\eta} h (\delta_{\tau} g_{2 x}^{1 x}, t, y) (\delta g_{2 z}^{1
    z} - \delta g_{2 x}^{1 x}) |\\
    & \lesssim & \| g_1 - g_2 \|_{C_T L^{\infty}} \left( \| g_1 \|_{C_T
    \CC^{\rho}} + \| g_2 \|_{C_T \CC^{\rho}} \right) \| h \|_{C_{\eta}^2 C_T
    \CC^{\gamma}} 2^{- \rho k} \sum_{q < k - 1} 2^{- \gamma q}\\
    &  & + \| g_1 - g_2 \|_{C_T \CC^{\rho}}  \| h \|_{C_{\eta}^1 C_T
    \CC^{\gamma}} 2^{- \rho k} \sum_{q < k - 1} 2^{- \gamma q} .
  \end{eqnarray*}
  With the same reasoning we can bound the norm of $\Pi_{\circ} (g_1, h) -
  \Pi_{\circ} (g_2, h)$.
\end{proof}

We will need the following time-smoothed nonlinear paraproduct
\begin{equation}
  \Pi_{\precprec} (g, h) (t, x) \assign \sum_i \int_{y, s} Q_{i, t} (s) P_{i,
  x} (y) (\Delta_i h (g (s, y), t, \cdot)) (x), \label{e:def-nonlinear}
\end{equation}
with $Q \in C^{\infty}_c (\mathbbm{R})$ a kernel with total mass $1$ as
specified in the introduction, and $Q_{i, t} (s) \assign 2^{2 i} Q (2^i (t -
s))$. In (\ref{e:def-nonlinear}) we use the convention that a continuous
function $t \mapsto g (t)$ on $\mathbbm{R}_+$ is extended to $\mathbbm{R}$ by
defining $g (t) = g (0)$ for $t \leqslant 0$. This preserves the H{\"o}lder
norms of index in $[0, 1]$. The modified nonlinear paraproduct enjoys similar
bounds to the regular one.

\begin{lemma}
  \label{l:nonlinear-est}Let $g \in C_T L^{\infty}$, $g \in [\lambda, 1]$ and
  $h \in C^1_{\eta} \LL_T^{\gamma}$ for $\gamma \in (0, 2)$. Then
  \[ \| \Pi_{\precprec} (g, h) \|_{\CC^{\gamma}_T} \lesssim \| h \|_{C_{\eta}
     \CC_T^{\gamma}} \quad \tmop{and} \quad \| \Pi_{\precprec} (g, h)
     \|_{\LL^{\gamma}_T} \lesssim \| h \|_{C_{\eta} \LL_T^{\gamma}} . \]
  Moreover, $\Pi_{\precprec} (g, h)$ is linear in $h$ and:
  \begin{eqnarray*}
    \| \Pi_{\precprec} (g_1, h) - \Pi_{\precprec} (g_2, h) \|_{\CC_T^{\gamma}}
    & \lesssim & \| g_1 - g_2 \|_{C_T L^{\infty}} \| h \|_{C^1_{\eta}
    \CC_T^{\gamma}},\\
    \| \Pi_{\precprec} (g_1, h) - \Pi_{\precprec} (g_2, h) \|_{\LL_T^{\gamma}}
    & \lesssim & \| g_1 - g_2 \|_{C_T L^{\infty}} \| h \|_{C^1_{\eta}
    \LL_T^{\gamma}} .
  \end{eqnarray*}
\end{lemma}

\begin{proof}
  The norm~$\| \Pi_{\precprec} (g, h) \|_{C_T \CC^{\gamma}}$ can be treated in
  the same way as in Lemma~\ref{l:decomp-est}. We estimate~$\| \Pi_{\precprec}
  (g, h) \|_{C_T^{\gamma / 2} \CC^0}$ as follows:
  \begin{eqnarray*}
    &  & \| \Delta_j \Pi_{\precprec} (g, h) (t_1) - \Delta_j \Pi_{\precprec}
    (g, h) (t_2) \|_{L^{\infty}}\\
    & \lesssim & \sup_x | \int_z K_{j, x} (z) \sum_{i \sim j} \int_{y, s}
    Q_{i, t_1} (s) P_{i, z} (y) [\Delta_i h (g (s, y), t_1, z) - \Delta_i h (g
    (s, y), t_2, z)] |\\
    &  & + \sup_x | \int_z K_{j, x} (z) \sum_{i \sim j} \int_{y, s} [Q_{i,
    t_1} (s) - Q_{i, t_2} (s)] P_{i, z} (y) \Delta_i h (g (s, y), t_2, z) |\\
    & \lesssim & \| h (\cdummy, t_1, \cdummy) - h (\cdummy, t_2, \cdummy)
    \|_{C_{\eta} \CC^0} + | t_1 - t_2 |^{\gamma / 2} \| h \|_{C_{\eta}
    \CC_T^{\gamma}} .
  \end{eqnarray*}
  The second inequality can be obtained easily with the same techniques used
  so far.
\end{proof}

\begin{remark}
  \label{rem:dumb}Using the Fourier support properties of the kernel $P_{q, x}
  (\cdummy)$ it is easy to see that it has mass $1$, i.e. $\int_y P_{q, x} (y)
  = 1$ $\forall x \in \mathbbm{T}^2, q \in \mathbbm{N}$. Then the paraproducts
  (\ref{e:nonlin-par-small}) and (\ref{e:def-nonlinear}) can be written as
  follows for a constant smooth function $g (t, x) = \bar{g}$, $\forall (t, x)
  \in \mathbbm{R} \times \mathbbm{T}^2$:
  \begin{eqnarray*}
    \Pi_{\prec} (\bar{g}, h) (t, x) & = & h (\bar{g}, t, x) =
    \Pi_{\diamondsuit} (\bar{g}, h),\\
    \Pi_{\precprec} (\bar{g}, h) (t, x) & = & h (\bar{g}, t, x) .
  \end{eqnarray*}
\end{remark}

\subsection{Nonlinear commutator}

The next technical ingredient is a commutator lemma between the non-linear
paraproduct of (\ref{e:def-nonlinear}) and the standard resonant product.
Since it will be needed below to analyse a term of the form $\Pi_{\precprec}
(g, h) \circ \Delta \Pi_{\precprec} (g, h)$, we will specialise our discussion
to this specific structure. Notice that in the following the various
space--time operators act pointwise in the parameter $\eta$, in the sense
that, for example:
\[ (h \circ \Delta h) (\eta, t, x) = (h (\eta, t, \cdot) \circ \Delta h (\eta,
   t, \cdot)) (x) . \]
\begin{lemma}
  \label{l:nonlin-comm}Define $\Lambda : C^{\infty} ([0, T], \mathbbm{T}^2)
  \times C_{\eta}^2 C^{\infty} ([0, T], \mathbbm{T}^2) \rightarrow C^{\infty}
  ([0, T], \mathbbm{T}^2)$ by
  \[ \Lambda (g, h) \assign [\Pi_{\precprec} (g, h) \circ \Delta
     \Pi_{\precprec} (g, h)] - \Pi_{\diamondsuit} (g, h \circ \Delta h) . \]
  Then for all \ $\rho \in (0, 1)$, $\gamma <\mathbbm{R}$, $\varepsilon > 0$
  such that $2 \gamma - 2 + \rho - \varepsilon > 0$ and $g \in [\lambda, 1]$,
  we have
  \begin{eqnarray*}
    \| \Lambda (g, h) \|_{\CC_T^{2 \gamma - 2 + \rho - \varepsilon}} &
    \lesssim & \lesssim \| g \|_{\LL_T^{\rho}} \| h \|_{C^1_{\eta}
    \CC^{\gamma}_T}^2,\\
    \| \Lambda (g_1, h) - \Lambda (g_2, h) \|_{\CC_T^{2 \gamma - 2 + \rho -
    \varepsilon}} & \lesssim & \| g_1 - g_2 \|_{C_T L^{\infty}} \left( \| g_1
    \|_{\LL^{\rho}_T} + \| g_2 \|_{\LL^{\rho}_T} \right) \| h \|^2_{C_{\eta}^2
    \CC^{\gamma}_T}\\
    &  & + \| g_1 - g_2 \|_{\LL^{\rho}_T} \| h \|^2_{C_{\eta}^1
    \CC^{\gamma}_T} .
  \end{eqnarray*}
  As a consequence $\Lambda$ can be uniquely extended to a locally Lipshitz
  function
  \[ \Lambda : \LL^{\rho}_T \times C_{\eta}^2 \CC_T^{\gamma} \rightarrow
     \CC_T^{2 \gamma - 2 + \rho - \varepsilon} . \]
\end{lemma}

\begin{proof}
  Let $A (t, z) \assign \Pi_{\precprec} (g, h) \circ \Delta \Pi_{\precprec}
  (g, h) (t, z)$ . We can approximate $A (t, z)$ with its value for a fixed $g
  = g (t, z)$, to obtain
  \begin{eqnarray}
    \Delta_q A (t, x) & = & \int_z K_{q, x} (z) (\Pi_{\precprec} (g, h) \circ
    \Delta \Pi_{\precprec} (g, h)) (t, z) \nonumber\\
    & = & \int_z K_{q, x} (z) (\Pi_{\precprec} (g (t, z), h) \circ \Delta
    \Pi_{\precprec} (g (t, z), h)) (t, z) \nonumber\\
    &  & + \int_z K_{q, x} (z) ((\Pi_{\precprec} (g, h) - \Pi_{\precprec} (g
    (t, z), h)) \circ \Delta \Pi_{\precprec} (g (t, z), h)) (t, z) 
    \label{e:Deltaeta2}\\
    &  & + \int_z K_{q, x} (z) (\Pi_{\precprec} (g, h) \circ \Delta
    (\Pi_{\precprec} (g, h) - \Pi_{\precprec} (g (t, z), h))) (t, z) . 
    \label{e:Deltaeta3}
  \end{eqnarray}
  with the remainders (\ref{e:Deltaeta2}) and (\ref{e:Deltaeta3}). To estimate
  the first term notice that
  \[ \Delta_j [\Delta \Pi_{\precprec} (g (t, z), h)] (t, z) = \Delta_j
     \Pi_{\precprec} (g (t, z), \Delta h) (t, z) \]
  and that, by Remark~\ref{rem:dumb}
  \begin{eqnarray*}
    \Delta_i \Pi_{\precprec} (g (t, z), h) (t, z) & = & \int_{z'} K_{i, z}
    (z') \Pi_{\precprec} (g (t, z), h) (t, z')\\
    & = & \int_{z'} K_{i, z} (z') h (g (t, z), t, z')\\
    & = & \Delta_i h (g (t, z), t, \cdummy) (z) .
  \end{eqnarray*}
  This yields
  \begin{eqnarray*}
    &  & (\Pi_{\precprec} (g (t, z), h) \circ \Delta \Pi_{\precprec} (g (t,
    z), h)) (t, z) = \sum_{i \sim j} \Delta_i \Pi_{\precprec} (g (t, z), h)\\
    & = & \sum_{i \sim j} \Delta_i h (g (t, z), t, \cdummy) (z) \Delta_j
    [\Delta h (g (t, z), t, \cdummy)] (z) = \Pi_{\diamondsuit} (g, h \circ
    \Delta h) (t, z) .
  \end{eqnarray*}
  We proceed then to estimate (\ref{e:Deltaeta2}) and (\ref{e:Deltaeta3}). We
  obtain
  \[ \int_{_z} K_{_{q, x}} (z) [(\Pi_{\precprec} (g, h) - \Pi_{\precprec} (g
     (t, z), h)) \circ \Delta \Pi_{\precprec} (g (t, z), h)] (t, z) \]
  \[ = \int_{_z} K_{_{q, x}} (z) \sum_{_{i \sim j \gtrsim q}} (\Delta_i
     \Pi_{\precprec} (g, h) (t, z) - \Delta_i \Pi_{\precprec} (g (t, z), h)
     (t, z)) \Delta_j \Delta \Pi_{\precprec} (g (t, z), h) (t, z) . \]
  Using Lemma~\ref{l:nonlinear-est} we have
  \[ | \Delta \Delta_j \Pi_{\precprec} (g (t, z), h) (t, z) | \lesssim 2^{(2 -
     \gamma) j} \| h \|_{C_{\eta} \CC^{\gamma}_T} . \]
  Lemma~\ref{l:para-increment} gives
  \[ | \Delta_i (\Pi_{\precprec} (g, h) - \Pi_{\precprec} (g (t, z), h)) (t,
     z) | \lesssim 2^{- (\gamma + \rho - \varepsilon) i} \| g
     \|_{\LL^{\rho}_T} \| h \|_{C_{\eta}^1 \CC^{\gamma}_{_T}}, \]
  and thus (\ref{e:Deltaeta2}) is bounded by $2^{- (2 \gamma + \rho - 2 -
  \varepsilon) q} \| g \|_{\LL^{\rho}_T} \| h \|^2_{\CC^{\gamma}_{_T}}$.
  
  We can easily bound (\ref{e:Deltaeta3}) in the same way, and this proves the
  first inequality. For the second inequality, Lemma~\ref{l:nonlinear-est}
  yields
  \[ | \Delta_j \Delta \Pi_{\precprec} (g_1, h) (t, z) - \Delta_j \Delta
     \Pi_{\precprec} (g_2, h) (t, z) | \lesssim 2^{(2 - \gamma) j} \| g_1 -
     g_2 \|_{C_T L^{\infty}} \| h \|_{C^1_{\eta} \CC^{\gamma}_T}, \]
  and using the second inequality of Lemma~\ref{l:para-increment} we obtain
  the desired bound.
  
  The extension of $\Lambda$ to $\LL^{\rho}_T \times C_{\eta}^2
  \LL_T^{\gamma}$ is standard (see e.g. the proof of the commutator lemma
  in~{\cite{gubinelli_paracontrolled_2015}}, Lemma~2.4). 
\end{proof}

\begin{lemma}
  \label{l:para-increment}Let us introduce the shortcut notation
  \[ \wp_i (g, h) (t, z) \assign \Delta_i \Pi_{\precprec} (g, h) (t, z) -
     \Delta_i \Pi_{\precprec} (g (t, z), h) (t, z) . \]
  Then, with the same assumptions of Lemma~\ref{l:nonlin-comm}, we have
  \[ | \wp_i (g, h) (t, z) | \lesssim 2^{(\varepsilon - \rho - \gamma) i} \| g
     \|_{\LL^{\rho}_T} \| h \|_{C_{\eta}^1 \CC^{\gamma}_{_T}} \]
  and
  \[ | \wp_i (g_1, h) (t, z) - \wp_i (g_2, h) (t, z) | \]
  \[ \lesssim 2^{(\varepsilon - \rho - \gamma) i} \left[ \| g_1 - g_2
     \|_{\LL^{\rho}_T} \| h \|_{C_{\eta}^1 \CC^{\gamma}_T} \right. \left. + \|
     g_1 - g_2 \|_{C_T L^{\infty}} \left( \| g_1 \|_{\LL^{\rho}_T} + \| g_2
     \|_{\LL^{\rho}_T} \right) \| h \|_{C_{\eta}^2 \CC^{\gamma}_T} \right] \]
\end{lemma}

\begin{proof}
  
  \begin{eqnarray*}
    &  & [\Delta_i \Pi_{\precprec} (g, h) - \Delta_i \Pi_{\precprec} (g (t,
    z), h)] (t, z)\\
    & = & \sum_{\underset{}{k \sim i}} \int_{_{\underset{s, \tau}{x, y}}}
    K_{i, z} (x) Q_{k, t} (s) P_{k - 1, x} (y) \partial_{\eta} \Delta_k h
    (\delta_{\tau} g^{s y}_{t z}, t, x) (\delta g^{s y}_{t y} + \delta g^{t
    y}_{t z})\\
    & \lesssim & \sum_{\underset{}{k \sim i}} \int_{_{\underset{s, \tau}{x,
    y}}} | K_{i, z} (x) Q_{k, t} (s) P_{k - 1, x} (y) | \| \partial_{\eta}
    \Delta_k h \|_{C_T L^{\infty}} | t - s |^{(\rho - \varepsilon) / 2} \| g
    \|_{C^{\rho / 2 - \varepsilon / 2}_T L^{\infty}}\\
    &  & + \sum_{\underset{}{k \sim i}} \int_{_{\underset{s, \tau}{x, y}}} |
    K_{i, z} (x) Q_{k, t} (s) P_{k - 1, x} (y) | \| \partial_{\eta} \Delta_k h
    \|_{C_T L^{\infty}} | y - z | \| g \|_{\CC^{\rho}_T}\\
    & \lesssim & 2^{- (\rho - \varepsilon) i} 2^{- \gamma i} \|
    \partial_{\eta} h \|_{\CC^{\gamma}_T} \left( \| g \|_{C^{\rho}_T \CC^0} +
    \| g \|_{\CC^{\rho}_T} \right)
  \end{eqnarray*}
  where we used the notation $\delta g^{s y}_{t z} = g (s, y) - g (t, z)$,
  $\delta_{\tau} g^{s y}_{t z} = g (t, z) + \tau (g (s, y) - g (t, z))$ and
  Lemma~\ref{l:interpol}. This proves the first bound.
  
  The second inequality can be obtained in the same way with the techniques
  already used here and in Lemma~\ref{l:decomp-est}.
\end{proof}

\subsection{Approximate paradifferential problem}\label{s:approx-paradiff}

In this section we construct an approximate solution to the equation
\begin{equation}
  (\partial_t - g \prec \Delta) u = f, \qquad u (0, \cdot) = 0,
  \label{e:paradiff}
\end{equation}
with data $f \in \CC^{\gamma - 2}$ and $g \in \LL^{\rho}_T$, for some fixed
$\rho, \gamma \in (0, 1)$. The idea is to obtain it via a certain class of
paradifferential operators. We introduce the operator $\LL$ acting on
functions of~$(\eta, t, x)$ by
\begin{equation}
  \left( \LL U \right) (\eta, t, x) \assign \partial_t U (\eta, t, x) - \eta
  \Delta U (\eta, t, x) . \label{e:def-Loper}
\end{equation}
We will also use the notation $\LL_{\eta_1} \assign \partial_t - \eta_1
\Delta$. \

Observe that if $u$ does not depend on $\eta$ we can define
\begin{equation}
  \Pi_{\prec} \left( g, \LL \right) u \assign \Pi_{\prec} \left( g, \LL u
  \right)  \label{e:paradiff-oper}
\end{equation}
and from definition (\ref{e:nonlin-par-small}) with $h = \LL u$ we obtain
$\Pi_{\prec} \left( g, \LL \right) u = \partial_t u - g \prec \Delta u$.

We can describe the commutation between the differential operator $\LL$ and
the paraproduct $\Pi_{\precprec} (g, \cdummy)$ via the following estimate:

\begin{lemma}
  \label{l:def-psi}Let $\rho \in (0, 1)$, $\gamma \in \mathbbm{R}$. Let $U \in
  C_{\eta}^2 \CC^{\gamma}_T$ and $g \in \LL^{\rho}_T$ such that $g \in
  [\lambda, 1]$. Define
  \[ \Psi (g, U) \assign R_1 + R_2 \]
  with $R_1$ and $R_2$ as in (\ref{e:rem-1}),~(\ref{e:rem-2}). Then for every
  $\varepsilon > 0$

  \begin{equation}
    \| \Psi (g, U) \|_{\CC_T^{\rho + \gamma - 2 - \varepsilon}} \lesssim (1 +
    \| g \|_{C_T L^{\infty}}) \| g \|_{\LL^{\rho}_T}^{} \| U \|_{C_{\eta}^1
    \CC^{\gamma}_{_T}}  \label{e:psi-est} .
  \end{equation}
  Moreover, $\Psi (g, U)$ is linear in $U$ and
  \[ \| \Psi (g_1, U) - \Psi (g_2, U) \|_{\CC_T^{\rho + \gamma - 2 -
     \varepsilon}} \lesssim \| g_1 - g_2 \|_{\LL^{\rho}_T} \left( 1 + \| g_1
     \|_{\LL^{\rho}_T} + \| g_2 \|_{\LL^{\rho}_T} \right) \| U \|_{C_{\eta}^2
     \CC^{\gamma}_{_T}} . \]
  In particular, we have
  \begin{equation}
    \Psi (g, U) = \Pi_{\precprec} \left( g, \LL U \right) - \Pi_{\prec} \left(
    g, \LL \right) \Pi_{\precprec} (g, U) \in C_T \CC^{\rho + \gamma - 2 -
    \varepsilon}  \label{e:quasidef-psi}
  \end{equation}
  whenever this expression makes sense.
\end{lemma}

\begin{proof}
  We start considering $g \in C^{\infty} ([0, T], \mathbbm{T}^2)$ and $U \in
  C_{\eta}^2 C^{\infty} ([0, T], \mathbbm{T}^2)$, and
  prove~$(\ref{e:quasidef-psi})$ in this setting. Notice that $\Pi_{\precprec}
  \left( g (t, y), \LL_{g (t, y)} U \right) = \LL U (g (t, y))$. As a
  consequence, we can estimate
  \begin{eqnarray*}
    &  & \Pi_{\precprec} \left( g, \LL U \right) (t, x) - \Pi_{\prec} \left(
    g, \LL \Pi_{\precprec} (g, U) \right) (t, x)\\
    & = & \Pi_{\precprec} \left( g, \LL U \right) (t, x) - \sum_k \int_y
    P_{k, x} (y) \left( \LL_{g (t, y)} \Delta_k \Pi_{\precprec} (g, U) \right)
    (t, x)\\
    & = & \Pi_{\precprec} \left( g, \LL U \right) (t, x) - \sum_k \int_y
    P_{k, x} (y) (\partial_t \Delta_k \Pi_{\precprec} (g, U)) (t, x)\\
    &  & + \sum_k \int_y P_{k, x} (y) g (t, y) (\Delta \Delta_k
    \Pi_{\precprec} (g, U)) (t, x)\\
    & = & \Pi_{\precprec} \left( g, \LL U - \partial_t U \right) (t, x)\\
    &  & + \sum_k \int_y P_{k, x} (y) g (t, y) (\Delta_k \Pi_{\precprec} (g,
    \Delta U)) (t, x)\\
    &  & + \sum_k \int_{} P_{k, x} (y) g (t, y) (\Delta_k [\Delta,
    \Pi_{\precprec} (g, \cdot)] U) (t, x)\\
    &  & - \sum_k \int_{} P_{k, x} (y) (\Delta_k [\partial_t, \Pi_{\precprec}
    (g, \cdot)] U) (t, x)
  \end{eqnarray*}
  with the commutators
  \begin{eqnarray*}
    {}[\Delta, \Pi_{\precprec} (g, \cdot)] U & \assign & \Delta
    \Pi_{\precprec} (g, U) - \Pi_{\precprec} (g, \Delta U),\\
    {}[\partial_t, \Pi_{\precprec} (g, \cdot)] U & \assign & \partial_t
    \Delta_k \Pi_{\precprec} (g, U) - \Delta_k \Pi_{\precprec} (g, \partial_t
    U) .
  \end{eqnarray*}
  We have
  \[ \Pi_{\precprec} \left( g, \LL U - \partial_t U \right) (t, x) + \sum_k
     \int_y P_{k, x} (y) g (t, y) (\Delta_k \Pi_{\precprec} (g, \Delta U)) (t,
     x) = R_1 (t, x) \]
  with the definition
  \[ R_1 (t, x) \assign \]
  \begin{equation}
    \sum_{k, i} \int_{_{\overset{y', s}{y, z}}} P_{k, x} (y) K_{k, x} (z)
    P_{i, z} (y') Q_{i, t} (s) [g (t, y) - g (s, y')] \Delta \Delta_i U (g (s,
    y'), t, z) \label{e:rem-1}
  \end{equation}
  and
  \[ \sum_k \int_y P_{k, x} (y) [g (t, y) (\Delta_k [\Delta, \Pi_{\precprec}
     (g, \cdot)] U) (t, x) - (\Delta_k [\partial_t, \Pi_{\precprec} (g,
     \cdot)] U) (t, x)] = R_2 (t, x) \]
  with the definition
  \[ R_2 (t, x) \assign \sum_{k, i} \int_{y, y', s} P_{k, x} (y) K_{k, x} (z)
     Q_{i, t} (s) g (t, y) \Delta P_{i, z} (y') \Delta_i U (g (s, y'), t, z)
  \]
  \[ + 2 \sum_{k, i} \int_{y, y', s} P_{k, x} (y) K_{k, x} (z) Q_{i, t} (s) g
     (t, y) \nabla P_{i, z} (y') \nabla \Delta_i U (g (s, y'), t, z) \]
  \begin{equation}
    - \sum_{k, i} \int_{y, y', s} P_{k, x} (y) K_{k, x} (z) \partial_t Q_{i,
    t} (s) P_{i, z} (y') \Delta_i U (g (s, y'), t, z) . \label{e:rem-2}
  \end{equation}

  Indeed:
  \begin{eqnarray}
    ([\partial_t, \Pi_{\precprec} (g, \cdot)] U) (t, x) & = & \sum_i \int_{y,
    s} (\partial_t Q_{i, t}) (s) P_{i, x} (y) (\Delta_i U (g (s, y), t, x)),
    \nonumber\\
    ([\Delta, \Pi_{\precprec} (g, \cdot)] U) (t, x) & = & \sum_i \int_{y, s}
    Q_{i, t} (s) \Delta P_{i, x} (y) (\Delta_i U (g (s, y), t, x)) 
    \label{eq:q-time}\\
    &  & + 2 \sum_i \int_{y, s} Q_{i, t} (s) \nabla P_{i, x} (y) (\nabla
    \Delta_i U (g (s, y), t, x)) . \nonumber
  \end{eqnarray}
  This shows that $(\ref{e:quasidef-psi})$ holds for smooth functions.
  
  With the techniques used in Lemma~\ref{l:para-increment} we can estimate
  \[ | \Delta_q R_1 (t, x) | \lesssim \sum_{k \sim q} (2^{- (\rho -
     \varepsilon) k} \| g \|_{C_T^{\rho / 2} \CC^0} + 2^{- \rho k} \| g
     \|_{\CC^{\rho}_T}) 2^{(2 - \gamma) k} \| U \|_{C_{\eta} \CC^{\gamma}_T} .
  \]
  By the spectral support properties of the commutators we have that
  \[ \| [\Delta, \Pi_{\precprec} (g, \cdot)] U \|_{\CC_T^{\gamma + \rho - 2}}
     \lesssim \| g \|_{\CC^{\rho}_T} \| U \|_{C^1_{\eta} \CC^{\gamma}_T}, \]
  and
  \[ \| \Delta_q [\partial_t, \Pi_{\precprec} (g, \cdot)] U \|_{C_T
     L^{\infty}} \lesssim (2^{(2 + \varepsilon - \rho - \gamma) q} \| g
     \|_{C_T^{\rho / 2} \CC^0} + 2^{(2 - \rho - \gamma) q} \| g
     \|_{\CC^{\rho}_T}) \| U \|_{C^1_{\eta} \CC^{\gamma}_T} . \]
  This yields
  \[ \| R_2 \|_{\CC_T^{\gamma + \rho - 2 - \varepsilon}} \lesssim (1 + \| g
     \|_{C_T L^{\infty}}) \| g \|_{\LL^{\rho}_T} \| U \|_{C^1_{\eta}
     \CC^{\gamma}_T} . \]
  We have so far proved (\ref{e:psi-est}) and then (\ref{e:quasidef-psi})
  follows by continuity. The local Lipshitz dependence on $g$ can be obtained
  via similar computations.
\end{proof}

\begin{remark}
  \label{r:approx-sol}If $f$ does not depend on $\eta$ we consider the
  parametric problem
  \begin{equation}
    (\partial_t - \eta \Delta) U_f (\eta, t) = f, \qquad U_f (\eta, 0) = 0,
    \qquad \eta \in [\lambda, 1], \label{e:def-U}
  \end{equation}
  which is solved by
  \[ U_f (\eta, t) = \int_0^t e^{\eta \Delta (t - s)} f \mathd s. \]
  Remark that
  \begin{eqnarray*}
    \partial_{\eta} U_f (\eta, t) & = & \int_0^t e^{\eta \Delta (t - s)} (t -
    s) \Delta f \mathd s \qquad \tmop{and}\\
    \partial_{\eta}^2 U_f (\eta, t) & = & \int_0^t e^{\eta \Delta (t - s)} (t
    - s)^2 \Delta^2 f \mathd s.
  \end{eqnarray*}
  We have, thanks to the well-known Schauder estimates of
  Lemma~\ref{l:schauder} (since $\eta \geqslant \lambda$):
  \begin{equation}
    \| U_f \|_{C_{\eta}^2 \LL^{\gamma}_T} \assign \sup_{n = 0, 1, 2}
    \sup_{\eta \in [\lambda, 1]} \| \partial_{\eta}^n U_f (\eta)
    \|_{\LL_T^{\gamma}} \lesssim \| f \|_{\CC^{\gamma - 2}_T}  \label{e:U-est}
  \end{equation}
  We define then
  \begin{equation}
    u (t, x) \assign \Pi_{\precprec} (g, U_f) (t, x) \label{e:def-u}
  \end{equation}
  and observe that $u (t, x)$ is an approximate solution of
  equation~(\ref{e:paradiff}), indeed
  \begin{eqnarray*}
    (\partial_t - g \prec \Delta) u & = & \Pi_{\prec} \left( g, \LL
    \Pi_{\precprec} (g, U_f) \right) = \Pi_{\precprec} \left( g, \LL U_f
    \right) - \Psi (g, U_f)\\
    & = & f - \Psi (g, U_f)
  \end{eqnarray*}
  and the estimation in Lemma~\ref{l:def-psi} together with the
  bound~(\ref{e:U-est}) yield immediately the following inequality:
  \begin{equation}
    \| \Psi (g, U_f) \|_{\CC^{\rho + \gamma - 2 - \varepsilon}_T} \lesssim \|
    g \|_{\LL^{\rho}_T} (1 + \| g \|_{C_T L^{\infty}}) \| f \|_{\CC^{\gamma -
    2}_T} . \label{e:psi-est2}
  \end{equation}
\end{remark}

\section{Paracontrolled Ansatz\label{s:paracontrolled}}

In order to give a meaning to the PDE in (\ref{e:paracontrolled-eq}) with
initial condition $u_0 \in \CC^{\alpha}$, our initial goal will be to get
informations on solutions $\theta = \theta (g)$ of the equation
\[ \partial_t \theta - g \prec \Delta \theta = \xi, \]
for a fixed $g \in \CC^{\alpha}_T$, $2 / 3 < \alpha < 1$, $g \in [\lambda,
1]$. Using the results of Section~\ref{s:approx-paradiff}, we consider to this
effect the parametric problem
\[ (\partial_t - \eta \Delta) \vartheta (\eta, t) = \xi, \]
for $\eta \in [\lambda, 1]$. We will consider the stationary solution of this
problem which has the form
\begin{equation}
  \vartheta (\eta, x) = \int_0^{\infty} e^{\eta \Delta s} \xi \mathd s = (-
  \eta \Delta)^{- 1} \xi \label{eq:def-theta}
\end{equation}
and in order for (\ref{eq:def-theta}) to be well defined we impose that the
noise $\xi$ has zero mean on $\mathbbm{T}^2$ (this is a simplifying assumption
which can be easily removed, e.g. at the price of adding a linear term to the
equation). We can control (\ref{eq:def-theta}) by bounding its
Littlewood-Paley blocks with a Bernstein lemma for distributions with
compactly supported Fourier transform ({\cite{bahouri_fourier_2011}},
Lemma~2.1) to obtain:
\begin{equation}
  \| \vartheta \|_{C_{\eta}^2 \LL^{\alpha}_T} = \| \vartheta \|_{C_{\eta}^2
  \CC^{\alpha}_T} \lesssim \| \xi \|_{\CC^{\alpha - 2}} .
  \label{eq:vartheta-est} 
\end{equation}
We define now for every $t \in [0, T]$
\[ \theta (t, x) \assign \Pi_{\precprec} (a (u), \vartheta) . \]
Thanks to Lemma~\ref{l:nonlinear-est} we have the bound $\| \theta
\|_{\LL^{\alpha}_{_T}} \lesssim \| \vartheta \|_{C_{\eta} \LL^{\alpha}_{_T}}
\lesssim \| \xi \|_{\CC^{\alpha - 2}}$ . We observe that this definition
together with Lemma~\ref{l:def-psi} gives
\[ \partial_t \theta - a (u) \prec \Delta \theta = \xi - \Psi (a (u),
   \vartheta) \]
with $\| \Psi (a (u), \vartheta) \|_{\CC^{2 \alpha - 2 - \varepsilon}_T}
\lesssim \| a (u) \|^2_{\LL^{\alpha}_T} \| \xi \|_{\CC^{\alpha - 2}_T}$ . We
expect then $\Psi (a (u), \vartheta)$ to be bounded in $\CC^{2 \alpha - 2 -
\varepsilon}_T$ for any $\varepsilon > 0$. At this point let us introduce the
Ansatz
\begin{equation}
  u = \theta + u^{\sharp} . \label{e:firstansatz}
\end{equation}
\begin{remark}
  Notice that we are not making any assumption on the existence of such~$u$,
  which is the subject of Section~\ref{s:fixpoint}. Our aim here is to find
  the equation that a couple $(u, u^{\sharp}) \in \CC_T^{\alpha} \times
  \CC_T^{2 \alpha}$ verifying (\ref{e:firstansatz}) must solve, in order for
  $u$ to solve (\ref{e:paracontrolled-eq}).
\end{remark}

Observe that
\begin{eqnarray*}
  \partial_t u - a (u) \prec \Delta u & = & (\partial_t - a (u) \prec \Delta)
  \theta + (\partial_t - a (u) \prec \Delta) u^{\sharp}\\
  & = & \xi + (\partial_t - a (u) \prec \Delta) u^{\sharp} - \Psi (a (u),
  \vartheta) .
\end{eqnarray*}
It follows that $u^{\sharp}$ must solve
\begin{equation}
  \left\{ \begin{array}{l}
    (\partial_t - a (u) \prec \Delta) u^{\sharp} = \Phi (u) + \Psi (a (u),
    \vartheta)\\
    u^{\sharp} (t = 0) = u^{\sharp}_0 \assign u_0 - \Pi_{\precprec} (a (u_0),
    \vartheta) (t = 0) \in \CC^{\alpha}
  \end{array} \right. \label{e:usharp-equadiff} \qquad
\end{equation}
with $\Phi (u) = a (u) \circ \Delta u + a (u) \succ \Delta u$, and if we can
make sense of the resonant term $a (u) \circ \Delta u$, it is reasonable to
expect $u^{\sharp} (t, \cdummy) \in \CC^{2 \alpha}$ $\forall t \in (0, T]$.
Indeed, take $U^{\sharp} \assign U_Q$ to be the solution of
\begin{equation}
  \LL U^{\sharp} (\eta) \assign (\partial_t - \eta \Delta) U^{\sharp} (\eta) =
  Q \qquad U^{\sharp} (\eta, t = 0) = 0 \label{e:def-Usharp} 
\end{equation}
for some $Q = Q (u^{\sharp})$ to be determined and $\eta \in [\lambda, 1]$.
Using again Lemma~\ref{l:def-psi} as shown in Remark~\ref{r:approx-sol} we
have
\[ (\partial_t - a (u) \prec \Delta) \Pi_{\precprec} (a (u), U^{\sharp}) = Q
   (u^{\sharp}) - \Psi (a (u), U^{\sharp}) . \]
For $\eta \in [\lambda, 1] $we define $\PP_t u_0^{\sharp} (\eta) \assign
e^{\eta \Delta t} u_0^{\sharp}$ so that $\LL \left( \PP_t u_0^{\sharp} \right)
= 0$, with $\LL$ as in (\ref{e:def-Loper}).

We set
\begin{equation}
  u^{\sharp} \assign \Pi_{\precprec} (a (u), U^{\sharp}) + \Pi_{\precprec}
  \left( a (u), \PP u_0^{\sharp} \right) . \label{eq:decomp-sharp}
\end{equation}
Taking
\[ Q (u^{\sharp}) : = \Phi (u) + \Psi (a (u), \vartheta) + \Psi (a (u),
   U^{\sharp}) + \Psi \left( a (u), \PP u_0^{\sharp} \right), \]
we obtain that $U^{\sharp}$ solves equation~(\ref{e:def-Usharp}) if and only
if $u^{\sharp}$ solves equation~(\ref{e:usharp-equadiff}). As we will see, $Q
(u^{\sharp}) (t)$ belongs to~$\CC^{2 \alpha - 2}$ $\forall t \in (0, T]$ but
not uniformly as $t \rightarrow 0$. However it belongs to $\CC^{\alpha - 2}$
uniformly as $t \rightarrow 0$.

It remains to control the resonant term $a (u) \circ \Delta u$ appearing in
$\Phi (u)$. We have
\[ a (u) \circ \Delta u = a (u) \circ \Delta \theta + a (u) \circ \Delta
   u^{\sharp} . \]
By paralinearization (see Theorem~\ref{th:paralinearisation}) $a (u) = a' (u)
\prec u + R_a (u)$ with $\| R_a (u) \|_{\CC^{2 \alpha}_T} \lesssim 1 + \| u
\|^2_{\CC^{\alpha}_T}$, and then
\[ a (u) \circ \Delta \theta = (a' (u) \prec u) \circ \Delta \theta + R_a (u)
   \circ \Delta \theta . \]
In order to use the commutator lemma (Lemma~\ref{lem:commutator}) we can
estimate $a' (u)$, recalling that $\alpha \in (0, 1)$, as
\[ \| a' (u) \|_{\CC^{\alpha}_T} \lesssim \| a'' \|_{L^{\infty}} \| u
   \|_{\CC^{\alpha}_T} \]
and write
\[ a (u) \circ \Delta \theta = a' (u) (u \circ \Delta \theta) + C (a' (u), u,
   \Delta \theta) + R_a (u) \circ \Delta \theta . \]
Then, Ansatz (\ref{e:firstansatz}) gives
\[ a (u) \circ \Delta \theta = a' (u) (\theta \circ \Delta \theta) + a' (u)
   (u^{\sharp} \circ \Delta \theta) + C (a' (u), u, \Delta \theta) + R_a (u)
   \circ \Delta \theta . \]
Summarizing, we have:
\begin{eqnarray*}
  \Phi (u) & = & a' (u) (\theta \circ \Delta \theta) + a (u) \succ \Delta u +
  a' (u) (u^{\sharp} \circ \Delta \theta)\\
  &  & + C (a' (u), u, \Delta \theta) + R_a (u) \circ \Delta \theta + a (u)
  \circ \Delta u^{\sharp}
\end{eqnarray*}
Thanks to the nonlinear commutator (Lemma~\ref{l:nonlin-comm}), we can
decompose the resonant term $\theta \circ \Delta \theta$ to obtain
\begin{eqnarray*}
  \Phi (u) & = & a (u) \succ \Delta u + a' (u) (u^{\sharp} \circ \Delta
  \theta) + C (a' (u), u, \Delta \theta) + R_a (u) \circ \Delta \theta\\
  &  & + a' (u) \Lambda (a (u), \vartheta) + a' (u) \Pi_{\diamondsuit} (a
  (u), \Theta_2) + a (u) \circ \Delta u^{\sharp}
\end{eqnarray*}
and $\Lambda (a (u), \vartheta) \in \CC^{3 \alpha - 2 - \varepsilon}_T$ if $u
\in \LL^{\alpha}_T$. Here we defined
\begin{equation}
  \Theta_2 (\eta, x) \assign (\vartheta \circ \Delta \vartheta) (\eta, x) =
  \sum_{i \sim j} \Delta_i \vartheta (\eta, \cdummy) (x) \Delta_j [\Delta
  \vartheta (\eta, \cdummy)] (x) \label{e:def-Theta}
\end{equation}
Finally, recalling the decomposition of $u^{\sharp}$ in two
terms~(\ref{eq:decomp-sharp}) we obtain
\[ \Phi (u) = a' (u) \Pi_{\diamondsuit} (a (u), \Theta_2) + \Phi_1 (u) +
   \Phi_2 (u) \]
where
\begin{eqnarray*}
  \Phi_1 (u) & \assign & a (u) \succ \Delta u + C (a' (u), u, \Delta \theta) +
  R_a (u) \circ \Delta \theta + a' (u) \Lambda (a (u), \vartheta)\\
  &  & + a' (u) (\Pi_{\precprec} (a (u), U^{\sharp}) \circ \Delta \theta) + a
  (u) \circ \Delta \Pi_{\precprec} (a (u), U^{\sharp}),\\
  \Phi_2 (u) & \assign & a' (u) \left( \Pi_{\precprec} \left( a (u), \PP
  u_0^{\sharp} \right) \circ \Delta \theta \right) + a (u) \circ \Delta
  \Pi_{\precprec} \left( a (u), \PP u_0^{\sharp} \right) .
\end{eqnarray*}
Thanks to Lemma~\ref{l:decomp-est} the terms $a' (u) \Pi_{\diamondsuit} (a
(u), \Theta_2)$ and $\Phi_1 (u)$ can be estimated in $\CC^{2 \alpha - 2}_T$,
provided $\Theta_2 \in C^2_{\eta} \CC^{2 \alpha - 2}_T$ (see
Section~\ref{s:renorm}). On the other hand the term $\Phi_2 (u) (t)$ can be
estimated in $\CC^{2 \alpha - 2}$ only for strictly positive times $t > 0$ due
to the lack of regularity of the initial condition $u_0^{\sharp}$ which a
priori lives only in $\CC^{\alpha}$.

Note moreover that the specific form of $\Phi$ allows to deduce that if we
replace $\Theta_2$ by $\tilde{\Theta}_2 = \Theta_2 - H$ with $H \in C_{\eta}^2
\CC^{2 \alpha - 2}_T$ then this is equivalent to consider an equation for $u$
of the form
\[ \partial_t u (t, x) - a (u (t, x)) \Delta u (t, x) = \xi (x) - a' (u (t,
   x)) H (a (u (t, x)), t, x) . \]
Let us resume this long discussion in the following theorem:

\begin{theorem}
  \label{th:identification}Assume that $\xi \in \CC^0, u_0 \in \CC^2, H \in
  C_{\eta}^2 \CC_T^0$. $u \in C^1_T \CC^2$ is the classical solution to the
  equation
  \begin{equation}
    \partial_t u (t, x) - a (u (t, x)) \Delta u (t, x) = \xi (x) - a' (u (t,
    x)) H (a (u (t, x)), t, x), \qquad u (0) = u_0, \label{eq:pde}
  \end{equation}
  up to time $T > 0$ if
  \[ u = \Pi_{\precprec} \left( a (u), \vartheta + U^{\sharp} + \PP
     u_0^{\sharp} \right), \]
  where $\vartheta$ is the solution to eq.~(\ref{eq:def-theta}) and
  $U^{\sharp}$ is the solution to the PDE
  \begin{equation}
    (\partial_t - \eta \Delta) U^{\sharp} (\eta) = F (u, U^{\sharp},
    u_0^{\sharp}) \qquad U^{\sharp} (\eta, 0) = 0 \qquad \eta \in [\lambda, 1]
    \label{eq:pde-sharp}
  \end{equation}
  with
  \begin{eqnarray*}
    F (u, U^{\sharp}, u_0) & = & a' (u) \Pi_{\nocomma \diamondsuit} (a (u),
    \Theta_2) + \Phi_1 (u) + \Phi_2 (u) + \Psi (a (u), \vartheta)\\
    &  & + \Psi (a (u), U^{\sharp}) + \Psi \left( a (u), \PP u_0^{\sharp}
    \right)
  \end{eqnarray*}
  and $\Theta_2 = \vartheta \circ \Delta \vartheta - H$.
\end{theorem}

\begin{definition}
  For any $\alpha \in \mathbbm{R}$ we define $\mathcal{X}^{\alpha} \subseteq
  C^2_{\eta} \CC^{\alpha} \times C^2_{\eta} \CC^{2 \alpha - 2}$ the closure of
  the image of the map
  \[ (\rho, H) \in C^2_{\eta} \CC^2 \times C^2_{\eta} \CC^0 \mapsto J (\rho,
     H) = (\rho, \rho \circ \Delta \rho - H) \in C^2_{\eta} \CC^2 \times
     C^2_{\eta} \CC^0 \]
  (in the topology of $C^2_{\eta} \CC^{\alpha} \times C^2_{\eta} \CC^{2 \alpha
  - 2}$).
\end{definition}

We call the elements in $\mathcal{X}^{\alpha}$ \tmtextit{enhanced noises}. In
the next section we will exploit the space $\mathcal{X}^{\alpha}$ for $2 / 3 <
\alpha < 1$ to solve equations~(\ref{eq:pde-sharp}) and (\ref{e:firstansatz}).
\

\section{Local wellposedness\label{s:fixpoint}}

The main result of this section is the local well--posedness for
equations~(\ref{e:firstansatz}) and~(\ref{eq:pde-sharp}) when $(\vartheta,
\Theta_2) \in \mathcal{X}^{\alpha}$ and $u_0 \in \CC^{\alpha}$ for $2 / 3 <
\alpha < 1$. This yields a unique solution to~(\ref{eq:pde}), thanks to
Theorem~\ref{th:identification}.

\begin{theorem}
  \label{t:well-posed}Let $\alpha > 2 / 3.$ Then for any $(\vartheta,
  \Theta_2) \in \mathcal{X}^{\alpha}$ and $u_0 \in \CC^{\alpha}$ there exists
  a time $T > 0$ depending only on $\| (\vartheta, \Theta_2)
  \|_{\mathcal{X}^{\alpha}}$ and $\| u_0 \|_{\alpha}$ up to which the system
  of equations~(\ref{e:firstansatz}) and~(\ref{eq:pde-sharp}) has a unique
  solution $(u, U^{\sharp}) \in \LL^{\alpha}_T \times C_{\eta}^2 \LL^{2
  \delta}_T$ for all $\delta < \alpha$ such that $2 \delta + \alpha > 2$. For
  any fixed $\tau > 0$ there exist a ball $B_{\tau} \subseteq \CC^{\alpha}
  \times \mathcal{X}^{\alpha}$ such that the solution map
  \[ \Sigma_{\tau} : (u_0, \vartheta, \Theta_2) \in B_{\tau} \mapsto (u,
     U^{\sharp}) \in \LL^{\alpha}_{\tau} \times C_{\eta}^2 \LL^{2
     \delta}_{\tau} \]
  is well defined and Lipshitz continuous in the data. 
\end{theorem}

\begin{remark}
  The proof is based on a Picard fixed point argument. In order to have a
  contraction map on a small time interval $[0, T]$, we carry on our analysis
  of $U^{\sharp}$ in the space $C_{\eta}^2 \LL^{2 \delta}_T \supset C_{\eta}^2
  \LL^{2 \alpha}_T$ and make use of the estimates of Lemma~\ref{l:schauder} to
  obtain a factor $T^{\varepsilon}$ for some (small) $\varepsilon > 0$.
\end{remark}

\begin{proof}[Theorem~\ref{t:well-posed}]
  Let $\mathcal{G}_T = \LL^{\alpha}_T \times C_{\eta}^2 \LL^{2 \delta}_T$. We
  introduce the map
  \[ \Gamma : (u, U^{\sharp}) \in \mathcal{G}_T \mapsto (\Gamma_u (u,
     U^{\sharp}), \Gamma_{U^{\sharp}} (u, U^{\sharp})) \in \mathcal{G}_T \]
  by
  \[ \Gamma_u (u, U^{\sharp}) \assign \Pi_{\precprec} (a (u), \vartheta) +
     \Pi_{\precprec} (a (u), \Gamma_{U^{\sharp}} (u, U^{\sharp})) +
     \Pi_{\precprec} \left( a (u), \PP u_0^{\sharp} \right) \]
  and
  \[ (\partial_t - \eta \Delta) \Gamma_{U^{\sharp}} (u, U^{\sharp}) (\eta) = F
     (u, U^{\sharp}, u_0^{\sharp}), \qquad \Gamma_{U^{\sharp}} (u, U^{\sharp})
     (\eta) (0) = 0, \qquad \eta \in [\lambda, 1], \]
  We will establish that this map is a contraction in the space $\mathcal{G}_T
  .$
  
  First, we have to show that there exists a ball $B \subset \mathcal{G}_T$
  such that $\Gamma (B) \subseteq B$. We have the bound $\left\| \PP
  u_0^{\sharp} \right\|_{C^1_{\eta} \LL^{\alpha}_T} \lesssim \| u^{\sharp}_0
  \|_{\CC^{\alpha}_T}$ . It is easy to obtain, using the estimates of
  Section~\ref{s:nonlinear} and Lemma~\ref{l:schauder}:
  \[ \| \int_0^T e^{- \eta \Delta (t - s)} \left[ \Phi_1 (u) + \Psi (a (u),
     \vartheta) + \Psi (a (u), U^{\sharp}) + \Psi \left( a (u), \PP
     u_0^{\sharp} \right) \right]_s \mathd s\|_{C^2_{\eta} \LL_T^{2 \delta}}
  \]
  \[ \lesssim T^{\kappa} \left( 1 + \| u \|_{\LL^{\alpha}_T} \right)^4 \left(
     1 + \| \xi \|_{\CC^{\alpha - 2}} \right)^2 \| u_0^{\sharp}
     \|_{\CC^{\alpha}} \left( 1 + \| U^{\sharp} \|_{C^2_{\eta} \LL^{2
     \delta}_T} \right) \]
  for some $\kappa > 0$.
  
  By the assumption that $(\vartheta, \Theta_2) \in \mathcal{X}^{\alpha}$ we
  deduce that there exists $M > 0$ such that $\| \Theta_2 \|_{C_{\eta}^2
  \CC_T^{2 \alpha - 2}} \leqslant M$. We have
  \[ \| \int_0^T e^{- \eta \Delta (t - s)} [a' (u) \Pi_{\nocomma \diamondsuit}
     (a (u), \Theta_2)]_s \mathd s\|_{C^2_{\eta} \LL_T^{2 \delta}} \lesssim
     T^{\alpha - \delta} \left( 1 + \| u \|_{\CC^{\alpha}_T} \right)^2 \|
     \Theta_2 \|_{C_{\eta}^2 \CC_T^{2 \alpha - 2}} . \]
  To bound the term $\Phi_2 (u)$ we observe that $\left\| \PP_t u^{\sharp}_0
  \right\|_{C_{\eta}^2 \CC^{2 \alpha}} \lesssim t^{- \frac{\alpha}{2}} \|
  u^{\sharp}_0 \|_{\CC^{\alpha}}$ thanks to Lemma~\ref{l:schauder}. This gives
  \[ \| \int_0^T e^{- \eta \Delta (t - s)} \Phi_2 (u)_s \mathd s\|_{C^2_{\eta}
     \LL_T^{2 \delta}} \lesssim T^{\alpha - \delta} \left( 1 + \| u
     \|_{\CC^{\alpha}_T} \right) \left( 1 + \| \xi \|_{\CC^{\alpha - 2}}
     \right) \| u^{\sharp}_0 \|_{\CC^{\alpha}} \]
  and then $\Gamma_{U^{\sharp}} (u, U^{\sharp})$ is bounded in $C^2_{\eta}
  \LL_T^{2 \delta}$ for $T$ small enough. We have also
  \begin{eqnarray*}
    \| \Gamma_u (u, U^{\sharp}) \|_{\LL^{\alpha}_T} & \lesssim & \| \xi
    \|_{\CC^{\alpha - 2}} + \| u^{\sharp}_0 \|_{\CC^{\alpha}} + \|
    \Gamma_{U^{\sharp}} (u, U^{\sharp}) \|_{C_{\eta} \LL_T^{\alpha}}\\
    & \lesssim & \| \xi \|_{\CC^{\alpha - 2}} + \| u^{\sharp}_0
    \|_{\CC^{\alpha}} + T^{\frac{2 \delta - \alpha}{2}} \| \Gamma_{U^{\sharp}}
    (u, U^{\sharp}) \|_{C^2_{\eta} \LL_T^{2 \delta}}
  \end{eqnarray*}
  and these bounds show that $\Gamma (B) \subseteq B$. The contractivity of
  $\Gamma_{U^{\sharp}} (u, U^{\sharp})$ can be obtained in the same way. Now
  consider $\Gamma_u (u, U^{\sharp})$: we have
  \begin{eqnarray*}
    &  & \| \Pi_{\precprec} (a (u_1), U_1^{\sharp}) - \Pi_{\precprec} (a
    (u_2), U_2^{\sharp}) \|_{\LL^{\alpha}_T}\\
    & \lesssim & T^{\frac{2 \delta - \alpha}{2}} \left( \| U^{\sharp}_1 -
    U^{\sharp}_2 \|_{C_{\eta} \LL^{2 \delta}} + \| u_1 - u_2 \|_{C_T
    L^{\infty}} \| U^{\sharp}_2 \|_{C_{\eta}^1 \LL^{2 \delta}} \right)
  \end{eqnarray*}
  while for the other terms in $\Gamma_u (u_1, U^{\sharp}_1) - \Gamma_u (u_2,
  U^{\sharp}_2)$ we remark that
  \[ \sup_{s \in [0, t]} \| u_{1, s} - u_0 - u_{2, s} + u_0 \|_{L^{\infty}}
     \lesssim t^{\varepsilon / 2} \| u_1 - u_2 \|_{C_{[0, t]}^{\varepsilon /
     2} L^{\infty}} . \]
  Then $\forall 0 < \varepsilon < \alpha$, using Lemma~\ref{l:nonlinear-est}
  and Lemma~\ref{l:interpol}:
  \begin{eqnarray*}
    \| \Pi_{\precprec} (a (u_1), \vartheta) - \Pi_{\precprec} (a (u_2),
    \vartheta) \|_{\LL^{\alpha}_T} & \lesssim & \| a (u_1) - a (u_2) \|_{C_T
    L^{\infty}} \| \vartheta \|_{C_{\eta}^1 \LL^{\alpha}_T}\\
    & \lesssim & \| u_1 - u_2 \|_{C_T L^{\infty}} \| \xi \|_{\CC^{\alpha -
    2}}\\
    & \lesssim & T^{\varepsilon / 2}  \| u_1 - u_2 \|_{C_T^{\varepsilon / 2}
    L^{\infty}} \| \xi \|_{\CC^{\alpha - 2}} \\
    & \lesssim & T^{\varepsilon / 2} \| u_1 - u_2 \|_{\LL^{\alpha}_T} \| \xi
    \|_{\CC^{\alpha - 2}} .
  \end{eqnarray*}
  With the same reasoning we estimate
  \begin{eqnarray*}
    &  & \left\| \Pi_{\precprec} \left( a (u_1), \PP u^{\sharp}_0 \right) -
    \Pi_{\precprec} \left( a (u_2), \PP u^{\sharp}_0 \right)
    \right\|_{\LL^{\alpha}_T}\\
    & \lesssim & T^{\varepsilon / 2} \| u_1 - u_2 \|_{C^{\varepsilon / 2}_T
    L^{\infty}} \left\| \PP u^{\sharp}_0 \right\|_{C_{\eta}^1
    \LL^{\alpha}_T}\\
    & \lesssim & T^{\varepsilon / 2} \| u_1 - u_2 \|_{\LL^{\alpha}_T} \|
    u^{\sharp}_0 \|_{\CC^{\alpha}_T}
  \end{eqnarray*}
  and then $\Gamma$ is a contraction for small times.
  
  The uniqueness of the solution $(u, U^{\sharp}) \in \LL^{\alpha}_T \times
  C_{\eta}^2 \LL^{2 \delta}_T$ and the Lipshitz continuity of the localized
  solution map $\Sigma_{\tau}$ can be proved along the same lines via standard
  arguments.
\end{proof}

\section{Renormalization\label{s:renorm}}

At this point we want to construct \tmtextit{an} enhanced noise $\Xi$
associated to the white noise $\xi$. Already in the standard setting of the
generalised PAM model with constant diffusion matrix, the construction of the
enhancement requires a renormalization since the resonant product $\vartheta
\circ \Delta \vartheta$ is not well defined.

Let $\psi \in \mathcal{S} (\mathbbm{T}^2)$ be a cutoff function and let
$\psi_{\varepsilon} (x) = \varepsilon^{- 2} \psi (x / \varepsilon)$. Then
define a regularised noise by $\xi_{\varepsilon} = \psi_{\varepsilon} \ast
\xi$ and let $\vartheta_{\varepsilon} = (- \eta \Delta)^{- 1}
\xi_{\varepsilon}$. Notice that
\begin{eqnarray*}
  H_{\varepsilon} (\eta) & \assign & \mathbbm{E} [\vartheta_{\varepsilon}
  (\eta, x) \circ \Delta \vartheta_{\varepsilon} (\eta, x)] =\mathbbm{E}
  [\vartheta_{\varepsilon} (\eta, x) \Delta \vartheta_{\varepsilon} (\eta,
  x)]\\
  & = & - \sum_{k \in \mathbbm{Z}^2 \backslash \{ 0 \}}
  \frac{\hat{\psi}_{\varepsilon} (k)^2}{\eta^2 | k |^2} = -
  \frac{\sigma_{\varepsilon}}{\eta^2}
\end{eqnarray*}
where
\[ \sigma_{\varepsilon} \assign \sum_{k \in \mathbbm{Z}^2 \backslash \{ 0 \}}
   \frac{\hat{\psi}_{\varepsilon} (k)^2}{| k |^2} \simeq | \log \varepsilon |
\]
as $\varepsilon \rightarrow 0$. Subtracting the diverging quantity
$H_{\varepsilon}$ to $\vartheta_{\varepsilon} \circ \Delta
\vartheta_{\varepsilon}$ and then taking the limit as $\varepsilon \rightarrow
0$ delivers a finite result.

\begin{theorem}
  \label{th:conv-white-noise}Take $\alpha < 1$ and let $\Xi_{\varepsilon} =
  (\xi_{\varepsilon}, \Xi_{2, \varepsilon}) : = (\xi_{\varepsilon},
  \vartheta_{\varepsilon} \circ \Delta \vartheta_{\varepsilon} -
  H_{\varepsilon})$. Then the family $(\Xi_{\varepsilon})_{\varepsilon}
  \subseteq \mathcal{X}^{\alpha}$ converges a.s. and in $L^p$ to a random
  element $\Xi = (\xi, \Xi_2) \in \mathcal{X}^{\alpha}$.
\end{theorem}

\begin{proof}
  The proof is a mild modification of the proof for
  PAM~{\cite{gubinelli_paracontrolled_2015}}. In order to establish the
  required $C^2_{\eta} \CC^{2 \alpha - 2}_T$ regularity for $\Xi_2$ we follow
  the computations for the case where the diffusion coefficient is constant.
  We only have to discuss the additional regularity in the parameter $\eta$.
  In order to do so observe that
  \[ \Xi_{2, \varepsilon} (\eta) = \sum_{i \sim j} \llbracket \Delta_i
     \vartheta_{\varepsilon} (\eta) \Delta_j \Delta \vartheta_{\varepsilon}
     (\eta) \rrbracket \]
  where $\llbracket \rrbracket$ denotes the Wick product with respect to the
  Gaussian structure of $\xi$. Then we have
  \[ \partial_{\eta} \Xi_{2, \varepsilon} (\eta) = \sum_{i \sim j} \llbracket
     \Delta_i \partial_{\eta} \vartheta_{\varepsilon} (\eta) \Delta_j \Delta
     \vartheta_{\varepsilon} (\eta) \rrbracket + \sum_{i \sim j} \llbracket
     \Delta_i \vartheta_{\varepsilon} (\eta) \Delta_j \Delta \partial_{\eta}
     \vartheta_{\varepsilon} (\eta) \rrbracket, \]
  and
  \begin{eqnarray*}
    \partial_{\eta}^2 \Xi_{2, \varepsilon} (\eta) & = & \sum_{i \sim j}
    \llbracket \Delta_i \partial_{\eta}^2 \vartheta_{\varepsilon} (\eta)
    \Delta_j \Delta \vartheta_{\varepsilon} (\eta) \rrbracket + \sum_{i \sim
    j} \llbracket \Delta_i \vartheta_{\varepsilon} (\eta) \Delta_j \Delta
    \partial_{\eta}^2 \vartheta_{\varepsilon} (\eta) \rrbracket\\
    &  & + \sum_{i \sim j} 2 \llbracket \Delta_i \partial_{\eta}
    \vartheta_{\varepsilon} (\eta) \Delta_j \Delta \partial_{\eta}
    \vartheta_{\varepsilon} (\eta) \rrbracket .
  \end{eqnarray*}

  Now the computations relative to the regularities of these additional
  stochastic objects are equivalent to those for the term $\Xi_{2,
  \varepsilon}$ where one or two instances of $\vartheta_{\varepsilon} (\eta)$
  are replaced by Gaussian fields of similar regularities of the form
  $\partial_{\eta} \vartheta_{\varepsilon} (\eta)$ and $\partial_{\eta}^2
  \vartheta_{\varepsilon} (\eta)$. A direct inspection of the proof contained
  in~{\cite{gubinelli_paracontrolled_2015}} allows us to deduce that we have
  almost sure $\CC^{2 \alpha - 2}$ regularity for these terms and also for
  random fields $\partial_{\eta}^n \Xi_{2, \varepsilon}$ for any finite $n$.
  This allows also to deduce that the random field is almost surely smooth in
  the parameter $\eta$. Similar computations allow to prove continuity in
  $\varepsilon$ for $\varepsilon > 0$. The rest of the proof is standard.
\end{proof}

In conclusion we see that in order to be able to use this convergence result
we need to modify our approximate PDE and consider instead
\[ \partial_t u_{\varepsilon} - a (u_{\varepsilon}) \Delta u_{\varepsilon} =
   \xi_{\varepsilon} - a' (u_{\varepsilon}) H_{\varepsilon} (a
   (u_{\varepsilon})) \]
which gives the renormalised equation (\ref{eq:ur-formulation-approx}).

Our well--posedness results for the paracontrolled formulation of this
equation together with the convergence result of
Theorem~\ref{th:conv-white-noise} allow to deduce that $u_{\varepsilon}
\rightarrow u$ in $\CC^{\delta}_T$ for any $2 / 3 < \delta < \alpha < 1$ and
that the limiting process $u$ satisfies a modified version of
eq.~(\ref{eq:ur-formulation}), namely
\[ \partial_t u - a (u) \diamond \Delta u = \xi, \qquad u (0) = u_0, \]
where $a (u) \diamond \Delta u$ denotes a \tmtextit{renormalized} diffusion
term given by
\begin{equation}
  a (u) \diamond \Delta u \assign a (u) \prec \Delta u + a' (u)
  \Pi_{\diamondsuit} (a (u), \Xi_2) + \Phi_1 (u) + \Phi_2 (u) .
  \label{eq:ren-diff}
\end{equation}
\section{Nonlinear source terms}\label{sec:gen-nonlinear}

Let us start by discussing the presence of a $u$ dependent r.h.s. in
eq.~(\ref{eq:ur-formulation}). We want to solve
\[ \partial_t u - a_1 (u) \Delta u = a_2 (u) \xi \]
where $a_1$ is a non-linear diffusion coefficient as before and $a_2 :
\mathbbm{R} \rightarrow \mathbbm{R}$ is another bounded function with
sufficiently many bounded derivatives. We rewrite this equation as
\[ \Pi_{\prec} \left( a (u), \LL \right) u = a_2 (u) \prec \xi + a_1 (u) \circ
   \Delta u + a_2 (u) \circ \xi + a_1 (u) \succ \Delta u + a_2 (u) \succ \xi
\]
where now $a (u) = (a_1 (u), a_2 (u))$ is a vector valued non-linearity. Since
we don't need $u$ to depend on any parameter $\eta = (\eta_1, \eta_2)$, we
have defined $\LL$ as
\[ \LL (\eta) \assign \partial_t - \eta_1 \Delta \]
and used the identity $\Pi_{\prec} \left( a (u), \LL \right) u = (\partial_t -
a_1 (u) \prec \Delta) u$, similarly to what we have done in
(\ref{e:paradiff-oper}).

Notice that the non--linear paraproduct can be extended trivially to the
vector valued case in such a way that, for example,
\[ \Pi_{\precprec} ((g_1, g_2), h) (t, x) = \sum_i \int_{y, s} Q_{i, t} (s)
   P_{i, x} (y) (\Delta_i h ((g_1 (s, y), g_2 (s, y)), t, \cdot)) (x) . \]
As before we make the Ansatz
\[ u = \Pi_{\precprec} (a (u), \vartheta) + u^{\sharp} \]
where now $\vartheta$ solves
\[ \LL (\eta) \vartheta (\eta) = (\partial_t - \eta_1 \Delta) \vartheta (\eta)
   = \eta_2 \xi, \]
for $\eta = (\eta_1, \eta_2) \in [\lambda, 1] \times [- L, L]$ where $L$ is a
large but fixed constant. The bounded domain is important to be able to have
uniform estimates and reuse the estimates proved above in the simple situation
of $\eta_2 = 1$. The solution of this equation is
\[ \vartheta (\eta, \cdot) = \eta_2 \int_0^{\infty} e^{\eta_1 \Delta s} \xi
   \mathd s = - \frac{\eta_2}{\eta_1} \Delta^{- 1} \xi . \]
Observe that
\[ \Pi_{\prec} \left( a (u), \LL \right) u = \Pi_{\prec} \left( a (u), \LL
   \right) \Pi_{\precprec} (a (u), \vartheta) + \Pi_{\prec} \left( a (u), \LL
   \right) u^{\sharp} \]
and recall that by Lemma~\ref{l:def-psi}
\[ \Pi_{\prec} \left( a (u), \LL \right) \Pi_{\precprec} (a (u), \vartheta) =
   \Pi_{\precprec} \left( a (u), \LL \vartheta \right) + \Psi (a (u),
   \vartheta) . \]
Now
\[ \left( \LL \vartheta \right) (\eta) = (\partial_t - \eta_1 \Delta)
   \vartheta (\eta, t, x) = \Xi (\eta), \qquad \eta = (\eta_1, \eta_2) \in
   [\lambda, 1] \times [- L, L] \]
with $\Xi (\eta) (t, x) = \eta_2 \xi (x)$ and then
\[ \Pi_{\precprec} \left( a (u), \LL \vartheta \right) = \Pi_{\precprec} (a
   (u), \Xi) = a_2 (u) \precprec \xi . \]
In conclusion
\[ \Pi_{\prec} \left( a (u), \LL \right) \Pi_{\precprec} (a (u), \vartheta) =
   a_2 (u) \precprec \xi + \Psi (a (u), \vartheta) \]
and the equation for $u^{\sharp}$ reads
\begin{eqnarray}
  \Pi_{\prec} \left( a (u), \LL u^{\sharp} \right) & = & a_1 (u) \circ \Delta
  u + a_2 (u) \circ \xi + [a_2 (u) \prec \xi - a_2 (u) \precprec \xi]
  \nonumber\\
  &  & + a_1 (u) \succ \Delta u + a_2 (u) \succ \xi - \Psi (a (u), \vartheta)
  \nonumber
\end{eqnarray}
where now all the terms on the r.h.s. can be considered remainder terms. Let
us just remark that the commutation term $a_2 (u) \prec \xi - a_2 (u)
\precprec \xi$ can be handled easily via Lemma~\ref{l:comm-paraproduct}. Of
course, the first two terms in the equation above require to be treated as
resonant terms. Note that, modulo terms of order $\CC^{3 \alpha - 2}_T$ (or
$\mathcal{E}^{\alpha / 2} \CC^{2 \alpha - 2}$ as defined in
Lemma~\ref{l:schauder}) the terms $a_1 (u) \circ \Delta u + a_2 (u) \circ \xi$
are equivalent to
\[ a'_1 (u) \Pi_{\diamondsuit} (a (u), \vartheta \circ \Delta \vartheta) +
   a'_2 (u) (\Pi_{\precprec} (a (u), \vartheta) \circ \xi) \]
and that by computations similar to those of the previous sections one can
prove that
\[ (\Pi_{\precprec} (a (u), \vartheta) \circ \xi) = \Pi_{\diamondsuit} (a (u),
   \vartheta \circ \xi) + \CC^{3 \alpha - 2}_T \]
so the resonant terms are comparable to the sum of the two terms
\[ a'_1 (u) \Pi_{\diamondsuit} (a (u), \vartheta \circ \Delta \vartheta) +
   a'_2 (u) \Pi_{\diamondsuit} (a (u), \vartheta \circ \xi) \]
which require renormalization of the form
\begin{equation}
  \frac{a'_1 (u) a_2 (u)^2}{a_1 (u)^2} \sigma_{\varepsilon} - \frac{a'_2 (u)
  a_2 (u)}{a_1 (u)} \sigma_{\varepsilon} \label{eq:renorm-non-linear-rhs}
\end{equation}
and the convergence follows with the same arguments of Section~\ref{s:renorm}.

We remark that the structure of the second renormalisation term, which is due
to the r.h.s. in the equation, is the same of that found by Bailleul,
Debussche and Hofmanov{\'a} in~{\cite{hofmanova}}.

\begin{remark}
  \label{r:vector-form}Our approach works straightforwardly for
  equation~(\ref{e:matrix-valued-coeff}), namely
  \[ \partial_t u (t, x) - a_{i j} (u (t, x)) \partial^2_{i j} u (t, x) = g (u
     (t, x)) \xi \]
  with $a : \mathbbm{R} \rightarrow M_2 (\mathbbm{R})$ such that $\sum_{i, j}
  a (u)_{i j} x_i x_j \geqslant C | x |^2$ $\forall x \in \mathbbm{R}^2$ for
  $C > 0$ and $\partial^2_{i j} \assign \frac{\partial^2}{\partial x_i
  \partial x_j}$.
  
  To see that, let $a (u) \assign (a_{i j} (u), g (u)) \in \mathbbm{R}^5$ and
  $\eta = (\eta_{i, j}, \eta_g) \in \mathbbm{R}^5$. Let $\LL (\eta) \assign
  \partial_t - \eta_{i j} \partial^2_{i j}$ and $\Xi (\eta) \assign \eta_g
  \xi$ with the uniform ellipticity condition $\sum_{i, j} \eta_{i j} x_i x_j
  \geqslant C | x |^2$ $\forall x \in \mathbbm{R}^2$. It is easy to verify
  that Lemma~\ref{l:def-psi} and Lemma~\ref{l:nonlin-comm} hold within this
  setting, just considering nonlinear paraproducts for functions depending on
  5 parameters. We have then:
  \[ u = \Pi_{\precprec} \left( a (u), \vartheta + U^{\sharp} + \PP
     u_0^{\sharp} \right) \]
  with $\vartheta (\eta)$ stationary solution of $\LL \vartheta (\eta) = \Xi
  (\eta)$, $\PP_t u_0^{\sharp} \assign e^{\eta_{i j} \partial^2_{i j} t}
  u_0^{\sharp}$ and $U^{\sharp} (\eta)$ which solves
  \[ \LL U^{\sharp} (\eta) = \Pi_{\diamondsuit} ((a (u), a' (u)), \Theta_1) +
     \Pi_{\diamondsuit} ((a (u), a' (u)), \Theta_2) + Q (u, U^{\sharp}) \]
  with $Q (u, U^{\sharp}) \in \CC^{2 \alpha - 2 - \varepsilon}$, $\Theta_1 \in
  C^2_{\eta} C^2_{\eta'} \CC^{2 \alpha - 2} = \vartheta (\eta) \circ \eta'_{i
  j} \partial^2_{i j} \vartheta (\eta)$, $\Theta_2 (\eta, \eta') = \vartheta
  (\eta) \circ \eta'_g \xi$ and $U^{\sharp} (t = 0) = 0$. Note that we can
  write $\vartheta$ as
  \[ \vartheta (\eta) = \eta_g \int_0^{\infty} e^{t \eta_{i j} \partial^2_{i
     j}} \xi \mathd t \qquad \hat{\vartheta} (k) = \eta_g \frac{\hat{\xi}
     (k)}{\eta_{i j} k_i k_j}, \quad k \in \mathbbm{Z}^2 \backslash \{ 0 \} .
  \]
  From the uniform ellipticity condition we have $\| \vartheta \|_{C_{\eta}^k
  \CC^{\alpha}} \lesssim \| \xi \|_{\CC^{\alpha}}$, and Schauder estimates
  analogous to those of Lemma~\ref{l:schauder} hold as well.
  
  Now consider the renormalization. We have
  \begin{eqnarray*}
    H_1^{\varepsilon} (\eta, \eta') \assign \mathbbm{E} (\Theta_1 (\eta,
    \eta')) & = & - \eta_g^2 \sum_{k \in \mathbbm{Z}^2 \backslash \{ 0 \}}
    \hat{\psi}_{\varepsilon} (k)^2 \frac{\sum_{i, j} \eta'_{i j} k_i
    k_j}{\left( \sum_{i, j} \eta_{i j} k_i k_j \right)^2},\\
    H_2^{\varepsilon} (\eta, \eta') \assign \mathbbm{E} (\Theta_2 (\eta,
    \eta')) & = & \eta_g \eta'_g \sum_{k \in \mathbbm{Z}^2 \backslash \{ 0 \}}
    \frac{\hat{\psi}_{\varepsilon} (k)^2}{\sum_{i, j} \eta_{i j} k_i k_j} .
  \end{eqnarray*}
  The convergence of $\Theta_1^{\varepsilon} - H_1^{\varepsilon}$,
  $\Theta_2^{\varepsilon} - H_2^{\varepsilon}$ in $C_{(\eta, \eta')}^k \CC^{2
  \alpha - 2} (\mathbbm{T}^2)$ can be obtained with the techniques used in
  {\cite{gubinelli_paracontrolled_2015}}, Section 5.2. 
\end{remark}

\section{Full generality}\label{sec:fully-generality}

Within the framework of the present work we are actually able to treat
equations of the form
\begin{equation}
  \partial_t u (t, x) - a_1 (u (t, x)) \Delta u (t, x) = \xi (a_2 (u (t, x)),
  x) \label{eq:general}
\end{equation}
where $\xi (\eta_2, x)$ is a Gaussian process with covariance
\[ \mathbbm{E} [\xi (\eta_2, x) \xi (\tilde{\eta}_2, \tilde{x})] = F (\eta_2,
   \tilde{\eta}_2) \delta (x - \tilde{x}) \]
where $F$ is a smooth covariance function. Let as before $2 / 3 < \alpha < 1$.
In this case we can take as a parametric equation
\[ \LL (\eta) \vartheta \assign \partial_t \vartheta (\eta, t, x) - \eta_1
   \Delta \vartheta (\eta, t, x) = \xi (\eta_2, x) \]
whose solution $\vartheta$ is a Gaussian process, smooth with respect to the
variable $\eta = (\eta_1, \eta_2)$ which we assume taking value in a compact
subset of $\mathbbm{R}^2$ for which $\eta_1 \geqslant \lambda > 0$ with fixed
$\lambda$. Letting $a (u) = (a_1 (u), a_2 (u))$ we can rewrite the l.h.s. of
eq.~(\ref{eq:general}) in the form
\[ \partial_t u - a_1 (u) \Delta u = \Pi_{\diamondsuit} \left( a (u), \LL u
   \right) \]
and the r.h.s. as
\[ \xi (a_2 (u (t, x)), x) = \Pi_{\diamondsuit} (a (u), \Xi) \]
where $\Xi (\eta, x) = \xi (\eta_2, x)$. Now we perform the paraproduct
decomposition to get
\begin{eqnarray*}
  \Pi_{\prec} \left( a (u), \LL u \right) - \Pi_{\prec} (a (u), \Xi) & = &
  \Pi_{\circ} (a (u), \Xi) + \Pi_{\circ} \left( a (u), \DF u \right)\\
  &  & + \Pi_{\succ} (a (u), \Xi) + \Pi_{\succ} \left( a (u), \DF u \right) .
\end{eqnarray*}
We have introduced here the parametric differential operator $\DF (\eta)
\assign \eta_1 \Delta$ for $\eta = (\eta_1, \eta_2)$.

Let $\PP_t (\eta) \assign e^{t \eta_1 \Delta}$ as before, and invoke the
paracontrolled Ansatz in the usual form
\[ u = \Pi_{\precprec} \left( a (u), \vartheta + U^{\sharp} + \PP
   u^{\sharp}_0 \right) . \]
Using that
\begin{eqnarray*}
  \Pi_{\prec} \left( a (u), \LL \Pi_{\precprec} \left( a (u), \vartheta +
  U^{\sharp} + \PP u^{\sharp}_0 \right) \right) & = & \Pi_{\precprec} \left( a
  (u), \LL \left( \vartheta + U^{\sharp} + \PP u_0^{\sharp} \right) \right)\\
  &  & + \Psi \left( a (u), \vartheta + U^{\sharp} + \PP u^{\sharp}_0 \right)
\end{eqnarray*}
and observing that we can take $\LL \vartheta = \Xi$ and that $\LL \PP
u^{\sharp}_0 = 0$ to get
\[ \Pi_{\precprec} \left( a (u), \LL U^{\sharp} \right) = F (u, U^{\sharp}) \]
where
\begin{eqnarray*}
  F (u, U^{\sharp}) & = & \Pi_{\circ} (a (u), \Xi) + \Pi_{\circ} \left( a (u),
  \DF u \right) + \Pi_{\succ} (a (u), \Xi) + \Pi_{\succ} \left( a (u), \DF u
  \right)\\
  &  & + [\Pi_{\prec} (a (u), \Xi) - \Pi_{\precprec} (a (u), \Xi)] - \Psi
  \left( a (u), \vartheta + U^{\sharp} + \PP u^{\sharp}_0 \right)
\end{eqnarray*}
which is solved by $U^{\sharp}$ satisfying
\[ \LL U^{\sharp} = F (u, U^{\sharp}) . \]
Indeed $\Pi_{\prec} (a (u), F (u, U^{\sharp})) = F (u, U^{\sharp})$, since $F
(u, U^{\sharp})$ does not depend on the additional parameter. Remark that the
term $\Pi_{\prec} (a (u), \Xi) - \Pi_{\precprec} (a (u), \Xi)$, which does not
appear in the simpler case, can be treated with
Lemma~\ref{l:comm-nonlin-para}.

It remains now to discuss the handling of the resonant products under the
paracontrolled assumption, namely $\Pi_{\circ} (a (u), \Xi)$ and $\Pi_{\circ}
\left( a (u), \DF u \right)$. Next lemma is a paralinearization result adapted
to our non-linear context.

\begin{lemma}
  Assume that $u \in \CC_T^{\rho}$ and $Z \in C_{\eta}^2 \CC^{\gamma}_T$ then
  if $\gamma + 2 \rho > 0$ we have
  \[ C (u, Z) \assign \Pi_{\circ} (a (u), Z) - u \circ \Pi_{\prec} ((a (u), a'
     (u)), \mathcal{D}Z) \in \CC_T^{\gamma + 2 \rho} \]
  where $\mathcal{D}_{} Z ((\eta, \eta'), t, x) \assign \sum_i \eta'_i
  \partial_{\eta_i} Z (\eta, t, x)$.
\end{lemma}

\begin{proof}
  
  \begin{eqnarray*}
    &  & \Pi_{\circ} (a (u), Z) (t, x) = \sum_{i \sim j} \int_{y, z} K_{i, x}
    (y) K_{j, x} (z) Z (a (u (t, y)), t, z)\\
    & = & \sum_{i \sim j, k} \int_{\tmscript{_{\tmscript{\begin{array}{l}
      y, z'\\
      z, z''
    \end{array}}}}} K_{i, x} (y) K_{j, x} (z) P_{k, z} (z'') K_{k, z} (z') Z
    (a (u (t, y)), t, z')\\
    & = & \sum_{i \sim j, k} \int_{\tmscript{_{\tmscript{\begin{array}{l}
      y, z'\\
      z, z''
    \end{array}}}}} K_{i, x} (y) K_{j, x} (z) P_{k, z} (z'') K_{k, z} (z')
    \times\\
    &  & \times [Z (a (u (t, y)), t, z') - Z (a (u (t, z'')), t, z')]\\
    & = & \sum_{i \sim j, k} \int_{\tmscript{_{\tmscript{\begin{array}{l}
      y, z'\\
      z, z''
    \end{array}}}}} K_{i, x} (y) K_{j, x} (z) P_{k, z} (z'') K_{k, z} (z')
    \times\\
    &  & \times [\sum_{\ell} a'_{\ell} (u (t, z'')) \delta u^{t y}_{t z''}
    \partial_{a_{\ell}} Z (a (u (t, z'')), t, z')]\\
    &  & + \sum_{\tmscript{_{\tmscript{\begin{array}{c}
      i \sim j\\
      k \sim j
    \end{array}}}}} \int_{\tmscript{_{\tmscript{\begin{array}{l}
      y, z'\\
      z, z''
    \end{array}}}}} K_{i, x} (y) K_{j, x} (z) P_{k, z} (z'') K_{k, z} (z')
    \times\\
    &  & \times O ((\delta u^{t y}_{t z''})^2) \partial_{\eta}^2 Z (a (u (t,
    y)), t, z')
  \end{eqnarray*}
  and observe that the first term is equal to $u \circ \Pi_{\prec} ((a (u), a'
  (u)), \mathcal{D}Z)$ while the second term can be easily estimated in
  $\CC_T^{\gamma + 2 \rho}$. 
\end{proof}

Using this result and Lemma~\ref{l:comm-nonlin-para} we can expand
\begin{eqnarray*}
  \Pi_{\circ} (a (u), \Xi) & = & u \circ \Pi_{\prec} ((a (u), a' (u)),
  \mathcal{D} \Xi) + \CC_T^{3 \alpha - 2}\\
  & = & \Pi_{\precprec} (a (u), \vartheta) \circ \Pi_{\prec} ((a (u), a'
  (u)), \mathcal{D} \Xi) + \CC_T^{3 \alpha - 2}\\
  & = & \Pi_{\diamondsuit} ((a (u), a' (u)), \vartheta \circ \mathcal{D} \Xi)
  + \CC_T^{3 \alpha - 2}
\end{eqnarray*}
and similarly, noting that
\begin{eqnarray*}
  \Pi_{\prec} \left( (a (u), a' (u)), \left( \mathcal{D} \DF \right) u \right)
  & = & \Pi_{\prec} \left( (a (u), a' (u)), \left( \mathcal{D} \DF \right)
  \Pi_{\precprec} (a (u), \vartheta) \right) + \CC_T^{3 \alpha - 2}\\
  & = & \Pi_{\prec} \left( (a (u), a' (u)), \left( \mathcal{D} \DF \right)
  \vartheta \right) + \CC_T^{3 \alpha - 2}
\end{eqnarray*}
where $\left( \mathcal{D} \DF \right) (\eta, \eta') = \eta_1' \Delta$, we have
\begin{eqnarray*}
  \Pi_{\circ} \left( a (u), \LL u \right) & = & u \circ \Pi_{\prec} \left( (a
  (u), a' (u)), \left( \mathcal{D} \DF \right) u \right) + \CC_T^{3 \alpha -
  2}\\
  & = & \Pi_{\precprec} (a (u), \vartheta) \circ \Pi_{\prec} \left( (a (u),
  a' (u)), \left( \mathcal{D} \DF \right) \vartheta \right) + \CC_T^{3 \alpha
  - 2}\\
  & = & \Pi_{\diamondsuit} \left( (a (u), a' (u)), \vartheta \circ \left(
  \mathcal{D} \DF \right) \vartheta \right) + \CC_T^{3 \alpha - 2}
\end{eqnarray*}

Finally the equation for $U^{\sharp}$ reads
\[ \LL U^{\sharp} = \Pi_{\diamondsuit} \left( (a (u), a' (u)), \vartheta \circ
   \mathcal{D} \Xi + \vartheta \circ \left( \mathcal{D} \DF \right) \vartheta
   \right) + \CC_T^{3 \alpha - 2} . \]
This can be solved essentially as we did in the simpler context. We see that
the general enhancement has the form
\[ \left( \xi, \vartheta \circ \mathcal{D} \Xi + \vartheta \circ \left(
   \mathcal{D} \DF \right) \vartheta \right) \]
which of course will require renormalization like we did before. In particular
\begin{eqnarray*}
  \left( \vartheta \circ \mathcal{D} \Xi + \vartheta \circ \left( \mathcal{D}
  \DF \right) \vartheta \right) (\eta, \eta') & = & \vartheta (\eta) \circ
  \eta'_2 \partial_{\eta_2} \xi (\eta_2, \cdot) + \vartheta (\eta) \circ
  \eta_1' \Delta \vartheta (\eta)\\
  & = & - \frac{\eta'_2}{\eta_1} (\Delta^{- 1} \xi (\eta_2, \cdot)) \circ
  \partial_{\eta_2} \xi (\eta_2, \cdot)\\
  &  & + \frac{\eta_1'}{\eta_1^2} (\Delta^{- 1} \xi (\eta_2, \cdot)) \circ
  \xi (\eta_2, \cdot)
\end{eqnarray*}
where we used that $\eta_1 \Delta \vartheta (\eta) = - \xi (\eta_2, \cdot)$.
Now observe that
\[ \mathbbm{E} [(\Delta^{- 1} \xi_{\varepsilon} (\eta_2, \cdot)) \circ
   \xi_{\varepsilon} (\eta_2, \cdot)] = - F (\eta_2, \eta_2)
   \sigma_{\varepsilon} \]
and that
\[ \mathbbm{E} [(\Delta^{- 1} \xi_{\varepsilon} (\eta_2, \cdot)) \circ
   \partial_{\eta_2} \xi_{\varepsilon} (\eta_2, \cdot)] = - (\partial_1 F)
   (\eta_2, \eta_2) \sigma_{\varepsilon} \]
with $\partial_1 F$ denoting the derivative with respect to the first entry.

In the end the renormalised enhanced noise is obtained as the limit in
$\mathcal{X}^{\alpha}$ of $(\xi_{\varepsilon}, \Xi_{2, \varepsilon})$ where
\[ \Xi_{2, \varepsilon} (\eta, \eta') = - \frac{\eta'_2}{\eta_1} (\Delta^{-
   1} \xi_{\varepsilon} (\eta_2, \cdot)) \circ \partial_{\eta_2}
   \xi_{\varepsilon} (\eta_2, \cdot) + \frac{\eta_1'}{\eta_1^2} (\Delta^{- 1}
   \xi_{\varepsilon} (\eta_2, \cdot)) \circ \xi_{\varepsilon} (\eta_2, \cdot)
   - H_{\varepsilon} (\eta, \eta') \]
with
\[ H_{\varepsilon} (\eta, \eta') = \frac{\eta'_2}{\eta_1} (\partial_1 F)
   (\eta_2, \eta_2) \sigma_{\varepsilon} - \frac{\eta_1'}{\eta_1^2} F (\eta_2,
   \eta_2) \sigma_{\varepsilon} . \]

We remark that if we take $F (\eta_2, \tilde{\eta}_2) = \eta_2 \tilde{\eta}_2$
we obtain again the situation treated in Section~\ref{sec:gen-nonlinear}.
Indeed in this case
\[ \Pi_{\diamondsuit} ((a (u), a' (u)), H_{\varepsilon}) = \frac{a_2' (u) a_2
   (u)}{a_1 (u)} \sigma_{\varepsilon} - \frac{a'_1 (u) a_2 (u)^2}{a_1 (u)^2}
   \sigma_{\varepsilon} . \]
which coincides with~(\ref{eq:renorm-non-linear-rhs}).

\begin{remark}
  Consider the more general equation (\ref{eq:general-model}), where the noise
  depends explicitly on time, e.g. with a covariance
  \[ \mathbbm{E} [\xi (\eta, t, x) \xi (\eta', t', x')] = F (\eta, \eta') Q (t
     - t', x - x') \]
  with $F$ a smooth function and $Q$ a distribution of parabolic regularity
  $\rho > - 4 / 3$. First note that the coefficient $a_1 (u) \in [\lambda, 1]$
  in front of the time derivative can be eliminated trivially by dividing.
  
  In order to handle the time dependence of the noise, the framework of this
  paper will still apply, provided we consider space--time paraproducts
  instead of paraproducts which act only on the space variable. This can be
  done exactly following the lines of the
  paper~{\cite{gubinelli_paracontrolled_2015}}, where time paraproducts are
  employed in the paracontrolled approach to solutions to SDE driven by
  gaussian signals.
  
  The constraint of regularity $\rho > - 4 / 3$ does allow to treat a noise
  which is white in time and smooth in space, but not a space--time white
  noise. It is well known that the first order paracontrolled approach on
  which the present paper is based does not allow to treat this kind of
  irregular signals in full generality.
\end{remark}

\appendix\section{Besov spaces and linear paraproducts}\label{appendix:besov}

In this Appendix we collect some classical results from harmonic analysis
needed in the paper. For a gentle introduction to Littlewood-Paley theory and
Besov spaces see the recent monograph~{\cite{bahouri_fourier_2011}}, where
most of our results are taken from. There the case of tempered distributions
on $\mathbbm{R}^d$ is considered. The Schauder estimates for the heat
semigroup (Lemma~\ref{l:schauder}) are classical and can be found
in~{\cite{gubinelli_paracontrolled_2015,gubinelli_lectures_2015}}.

\

Fix $d \in \mathbbm{N}$ and denote by $\mathbbm{T}^d = (\mathbbm{R}/ (2 \pi
\mathbbm{Z}))^d$ the $d$-dimensional torus. We focus here on distributions and
SPDEs on the torus, but everything in this Appendix applies mutatis-mutandis
on the full space $\mathbbm{R}^d$, see~{\cite{gubinelli_paracontrolled_2015}}.
The space of distributions $\CD' = \CD' (\mathbbm{T}^d)$ is defined as the set
of linear maps $f$ from $C^{\infty} = C^{\infty} (\mathbbm{T}^d, \mathbbm{C})$
to $\mathbbm{C}$, such that there exist $k \in \mathbbm{N}$ and $C > 0$ with
\[ | \langle f, \varphi \rangle | \assign | f (\varphi) | \leqslant C \sup_{|
   \mu | \leqslant k} \| \partial^{\mu} \varphi \|_{L^{\infty}
   (\mathbbm{T}^d)} \]
for all $\varphi \in C^{\infty}$. In particular, the Fourier transform $\CF f
: \mathbbm{Z}^d \rightarrow \mathbbm{C}$, $\CF f (k) = \langle f, e^{- i k
\cdummy} \rangle$, is defined for all $f \in \CD'$, and it satisfies $\left|
\CF f (k) \right| \leqslant | P (k) |$ for a suitable polynomial $P$. We will
also write $\hat{f} (k) = \CF f (k)$. Conversely, if $(g (k))_{k \in
\mathbbm{Z}^d}$ is at most of polynomial growth, then its inverse Fourier
transform
\[ \CF^{- 1} g = (2 \pi)^{- d} \sum_{k \in \mathbbm{Z}^d} e^{i \langle k,
   \cdummy \rangle} g (k) \]
defines a distribution, and we have $\CF^{- 1} \CF f = f$ as well as $\CF
\CF^{- 1} g = g$. To see this, it suffices to note that the Fourier transform
of $\varphi \in C^{\infty}$ decays faster than any rational function (we say
that it is of \tmtextit{rapid decay}). Indeed, for $\mu \in \mathbbm{N}^d_0$
we have $| k^{\mu} \hat{g} (k) | = \left| \CF (\partial^{\mu} g) (k) \right|
\leqslant \| \partial^{\mu} g \|_{L^1 (\mathbbm{T}^d)}$ for all $k \in
\mathbbm{Z}^d$. As a consequence we get the Parseval formula $\left\langle f,
\overline{\varphi} \right\rangle = (2 \pi)^{- d} \sum_k \hat{f} (k)
\overline{\hat{\varphi} (k)}$ for $f \in \CD'$ and $\varphi \in C^{\infty}$.

Linear maps on $\CD'$ can be defined by duality: if $A : C^{\infty}
\rightarrow C^{\infty}$ is such that for all $k \in \mathbbm{N}$ there exists
$n \in \mathbbm{N}$ and $C > 0$ with $\sup_{| \mu | \leqslant k} \|
\partial^{\mu} (A \varphi) \|_{L^{\infty}} \leqslant C \sup_{| \mu | \leqslant
n} \| \partial^{\mu} \varphi \|$, then we set $\langle^t A f, \varphi \rangle
= \langle f, A \varphi \rangle$. Differential operators are defined by
$\langle \partial^{\mu} f, \varphi \rangle = (- 1)^{| \mu |} \langle f,
\partial^{\mu} \varphi \rangle$. If $\varphi : \mathbbm{Z}^d \rightarrow
\mathbbm{C}$ grows at most polynomially, then it defines a \tmtextit{Fourier
multiplier}
\[ \varphi (\mathD) f = \CF^{- 1} \left( \varphi \CF f \right), \]
which gives us a distribution $\varphi (\mathD) f \in \CD'$ for every $f \in
\CD'$.

\

Littlewood-Paley blocks give a decomposition of any distribution on $\CD'$
into an infinite series of smooth functions.

\begin{definition}
  A \tmtextit{dyadic partition of unity} consists of two nonnegative radial
  functions $\chi, \rho \in C^{\infty} (\mathbbm{R}^d, \mathbbm{R})$, where
  $\rho$ is supported in a ball $\CB = \{ | x | \leqslant c \}$ and $\rho$ is
  supported in an annulus $\CA = \{ a \leqslant | x | \leqslant b \}$ for
  suitable $a, b, c > 0$, such that
  \begin{enumerate}
    \item $\chi + \sum_{j \geqslant 0} \rho (2^{- j} \cdummy) \equiv 1$ and
    
    \item \label{property-2-dyadic}$\chi \rho (2^{- j} \cdummy) \equiv 0$ for
    $j \geqslant 1$ and $\rho (2^{- i} \cdummy) \rho (2^{- j} \nosymbol
    \cdummy) \equiv 0$ for all $i, j \geqslant 0$ with $| i - j | > 1$.
  \end{enumerate}
  We will often write $\rho_{- 1} = \chi$ and $\rho_j = \rho (2^{- j}
  \cdummy)$ for $j \geqslant 0$.
\end{definition}

Dyadic partitions of unity exist, as shown in~{\cite{bahouri_fourier_2011}}.
The reason for considering smooth partitions rather than indicator functions
is that indicator functions do not have good Fourier properties. We fix a
dyadic partition of unity $(\chi, \rho)$ and define the dyadic blocks
\[ \Delta_j f = \rho_j (\mathD) f = \CF^{- 1} (\rho_j \hat{f}), \quad j
   \geqslant - 1. \]
We also use the notation
\[ S_j f = \sum_{i < j - 1} \Delta_i f \]
and notice that
\[ \Delta_j f (x) = \int K_{j, x} (y) f (y) \mathd y, \qquad S_j f (x) = \int
   P_{j, x} (y) f (y) \mathd y \]
with $K_{j, x} (y) = 2^{dj} K (2^j (x - y))$, $P_{j, x} (y) = \sum_{i < j - 1}
K_{i, x} (y)$, $K$ radial with zero mean.

Every dyadic block has a compactly supported Fourier transform and therefore
it belongs to~$C^{\infty}$. It is easy to see that $f = \sum_{j \geqslant - 1}
\Delta_j f = \lim_{j \rightarrow \infty} S_j f$ for all $f \in \CD'$.

For $\alpha \in \mathbbm{R}$, the H{\"o}lder-Besov space $\CC^{\alpha}$ is
given by $\CC^{\alpha} = B^{\alpha}_{\infty, \infty} (\mathbbm{T}^d,
\mathbbm{R})$, where for $p, q \in [1, \infty]$ we define
\[ B^{\alpha}_{p, q} = B^{\alpha}_{p, q} (\mathbbm{T}^d, \mathbbm{R}) = \{f
   \in \DD' : \|f\|_{B^{\alpha}_{p, q}} = (\sum_{j \geqslant - 1} (2^{j
   \alpha} \| \Delta_j f\|_{L^p})^q)^{1 / q} < \infty\}, \]
with the usual interpretation as $\ell^{\infty}$ norm in case $q = \infty$.
Then $B^{\alpha}_{p, q}$ is a Banach space and while the norm $\lVert \cdummy
\rVert_{B^{\alpha}_{p, q}}$ depends on $(\chi, \rho)$, the space
$B^{\alpha}_{p, q}$ does not, and any other dyadic partition of unity
corresponds to an equivalent norm.

If $\alpha \in (0, \infty) \setminus \mathbbm{N}$, then $\CC^{\alpha}$ is the
space of $\lfloor \alpha \rfloor$ times differentiable functions whose partial
derivatives of order $\lfloor \alpha \rfloor$ are ($\alpha - \lfloor \alpha
\rfloor$)-H{\"o}lder continuous (see page~99
of~{\cite{bahouri_fourier_2011}}). Note however, that for $k \in \mathbbm{N}$
the space $\CC^k$ is strictly larger than $C^k$, the space of $k$ times
continuously differentiable functions.

The following lemma gives useful characterisation of Besov regularity for
functions that can be decomposed into pieces which are localized in Fourier
space.

\begin{lemma}
  \label{lem: Besov regularity of series}{\tmdummy}
  
  \begin{enumeratenumeric}
    \item Let $\CA$ be an annulus, let $\alpha \in \mathbb{\mathbbm{R}}$, and
    let $(u_j)$ be a sequence of smooth functions such that $\CF u_j$ has its
    support in $2^j \CA$, and such that $\|u_j \|_{L^{\infty}} \lesssim 2^{- j
    \alpha}$ for all $j$. Then
    \[ u = \sum_{j \geqslant - 1} u_j \in \CC^{\alpha} \qquad \tmop{and}
       \qquad \|u\|_{\alpha} \lesssim \sup_{j \geqslant - 1} \{2^{j \alpha}
       \|u_j \|_{L^{\infty}} \} . \]
    \item Let $\CB$ be a ball, let $\alpha > 0$, and let $(u_j)$ be a sequence
    of smooth functions such that $\CF u_j$ has its support in $2^j \CB$, and
    such that $\|u_j \|_{L^{\infty}} \lesssim 2^{- j \alpha}$ for all $j$.
    Then
    \[ u = \sum_{j \geqslant - 1} u_j \in \CC^{\alpha} \qquad \tmop{and}
       \qquad \|u\|_{\alpha} \lesssim \sup_{j \geqslant - 1} \{2^{j \alpha}
       \|u_j \|_{L^{\infty}} \} . \]
  \end{enumeratenumeric}
\end{lemma}

\

The Bernstein inequalities of the next lemma are extremely useful when
dealing with functions with compactly supported Fourier transform.

\begin{lemma}
  \label{lemma:Bernstein}Let $\CA$ be an annulus and let $\CB$ be a ball. For
  any $k \in \mathbbm{N}_0$, $\lambda > 0$, and $1 \leqslant p \leqslant q
  \leqslant \infty$ we have that
  \begin{enumeratenumeric}
    \item if $u \in L^p$ is such that $\tmop{supp} (\CF u) \subseteq \lambda
    \CB$, then
    \[ \max_{\mu \in \mathbb{N}^d : | \mu | = k} \| \partial^{\mu} u\|_{L^q}
       \lesssim_k \lambda^{k + d \left( \frac{1}{p} - \frac{1}{q} \right)}
       \|u\|_{L^p} ; \]
    \item if $u \in L^p$ is such that $\tmop{supp} (\CF u) \subseteq \lambda
    \CA$, then
    \[ \lambda^k \|u\|_{L^p} \lesssim_k \max_{\mu \in \mathbb{N}^d : | \mu | =
       k} \| \partial^{\mu} u\|_{L^p} . \]
  \end{enumeratenumeric}
\end{lemma}

We recall some standard heat kernel estimations (see Chapter 2 of
{\cite{bahouri_fourier_2011}} and
{\cite{gubinelli_paracontrolled_2015,gubinelli_lectures_2015}}).

\begin{lemma}[Schauder estimates]
  \label{l:schauder}Let $V_t = \int_0^t e^{\eta (t - s) \Delta} v_s \mathd s$
  and $\PP_t u_0 = e^{\eta \Delta t} u_0$, with $\eta \geqslant \lambda$. We
  define $\LL^{\alpha}_T$ and $C_{\eta}^k \LL^{\alpha}_T$, $C_{\eta}^k
  \CC^{\alpha}_T$ for $k \in \mathbbm{N}$ as in
  (\ref{e:def-Lnorm}),~(\ref{e:def-Cetanorm}) and introduce the norm
  \[ \| v \|_{\mathcal{E}^{\delta}_T \CC^{\alpha}} = \sup_{t \in [0, T]}
     t^{\delta} \| v (t, \cdummy) \|_{\CC^{\alpha}} . \]
  Then for any $\gamma \in [0, 1)$ and $\alpha \in \mathbbm{R}$:
  \begin{eqnarray*}
    \sup_{t \in [0, T]} t^{\gamma} \| V_t \|_{C_{\eta}^k \CC^{\alpha - 2
    \beta}} & \lesssim & \sup_{t \in [0, T]} t^{\gamma + \beta} \| v_t
    \|_{\CC^{\alpha - 2}} \quad \forall \beta \in [0, 1)\\
    \| V \|_{C_{\eta}^k \LL^{\alpha - 2 \beta}_T} & \lesssim & T^{\beta} \| v
    \|_{\CC^{\alpha - 2}_T} \quad \forall \beta \in [0, 1)\\
    \| V \|_{C_{\eta}^k \LL_T^{\alpha - 2 \beta}} & \lesssim & T^{\beta -
    \delta} \| v \|_{\mathcal{E}^{\delta}_{_T} \CC^{\alpha - 2}} \quad \forall
    \beta \in [0, 1), \forall \delta \in [0, \beta]\\
    \| V \|_{C_{\eta}^k \LL_T^{\gamma}} & \lesssim & T^{\frac{\rho -
    \gamma}{2} + 1 - \delta} \| v \|_{\mathcal{E}^{\delta}_{_T} \CC^{\rho}}\\
    &  & \forall \rho \in [\gamma - 2, \gamma), \forall \delta \in [0, (\rho
    - \gamma) / 2 + 1]\\
    \left\| \PP_t u_0 \right\|_{C_{\eta}^k \CC^{\alpha}} & \lesssim & t^{-
    \frac{\alpha - \beta}{2}} \| u_0 \|_{\CC^{\beta}} \quad \forall \beta <
    \alpha\\
    \left\| \PP_t u_0 - \PP_s u_0 \right\|_{C_{\eta}^k \CC^{\alpha}} &
    \lesssim & s^{- \delta} | t - s |^{\frac{\beta - \alpha}{2} + \delta} \|
    u_0 \|_{\CC^{\beta}_{}}\\
    &  & \forall t \neq s \in \mathbbm{R}^+, \beta \leqslant \alpha + 2,
    \delta \in [0, 1 - (\beta - \alpha) / 2]
  \end{eqnarray*}
\end{lemma}

We need the following interpolation lemma:

\begin{lemma}
  \label{l:interpol} Let $\gamma \in (0, 2)$, $0 < \varepsilon < \gamma$ and
  $u \in \LL^{\gamma}_T$ . Then
  \[ \| u \|_{C_T^{\gamma / 2 - \varepsilon / 2} L^{\infty}} \lesssim \| u
     \|_{\LL^{\gamma}_T} \]
\end{lemma}

\begin{proof}
  
  \begin{eqnarray*}
    \sup_{s \neq t} \frac{\| u_t - u_s \|_{L^{\infty}}}{| t - s |^{\gamma / 2
    - \varepsilon / 2}} & \leqslant & \sup_{s \neq t} [\sum_{i \leqslant n}
    \frac{\| \Delta_i u_t - \Delta_i u_s \|_{L^{\infty}}}{| t - s |^{\gamma /
    2 - \varepsilon / 2}} + \sum_{i > n} \frac{\| \Delta_i u_t - \Delta_i u_s
    \|_{L^{\infty}}}{| t - s |^{\gamma / 2 - \varepsilon / 2}}]
  \end{eqnarray*}
  and choosing $2^{- n - 1} \leqslant | t - s |^{1 / 2} \leqslant 2^{- n}$ we
  obtain
  \begin{eqnarray*}
    \sum_{i < n} \frac{\| \Delta_i u_t - \Delta_i u_s \|_{L^{\infty}}}{| t - s
    |^{\gamma / 2 - \varepsilon / 2}} & \lesssim & \| u \|_{C_T^{\gamma / 2}
    \CC^0} \sum_{i \leqslant n} | t - s |^{\varepsilon / 2}\\
    \sum_{i \geqslant n} \frac{\| \Delta_i u_t - \Delta_i u_s
    \|_{L^{\infty}}}{| t - s |^{\gamma / 2 - \varepsilon / 2}} & \lesssim & \|
    u \|_{\CC^{\gamma}_T} \sum_{i > n} 2^{- \gamma i} 2^{- (\gamma +
    \varepsilon) n}
  \end{eqnarray*}
  and this yields the result.
\end{proof}

Terms of the type $\| a (u (t, x)) \|_{\CC^0}$ with $a : \mathbbm{R}
\rightarrow \mathbbm{R}$ cannot be estimated directly with their H{\"o}lder
norm. In the following lemma we prove some bounds used in
Section~\ref{s:fixpoint}.

\begin{lemma}
  Let $a \in C^3_b$ uniformly bounded and $u \in \LL^{\alpha}_T =
  \CC^{\alpha}_T \cap C_T^{\alpha / 2} \CC^0$, then
  \begin{eqnarray*}
    \| a (u) \|_{\LL^{\alpha}_T} & \lesssim & \| a \|_{L^{\infty}} + \| a'
    \|_{L^{\infty}} (\| u \|_{C_T^{\alpha / 2} \CC^0} + \| u
    \|_{\CC^{\alpha}_{_T}}) + \| a'' \|_{L^{\infty}} \| u \|_{C_T^{\alpha / 2}
    \CC^0} \| u \|_{\CC^{\alpha}_{_T}}\\
    & \lesssim & 1 + \| u \|_{\LL^{\alpha}_T} + \| u \|^2_{\LL_T^{\alpha}},
  \end{eqnarray*}
  \begin{eqnarray*}
    \| a (u_1) - a (u_2) \|_{\LL_T^{\alpha}} & \lesssim & \| u_1 - u_2
    \|_{\LL^{\alpha}_T} \left( 1 + \| u_1 \|_{\CC^{\alpha}_T} + \| u_2
    \|_{\CC^{\alpha}_T} \right)^2,\\
    \| a (u_1) - a (u_2) \|_{\CC_T^0} & \lesssim & \| u_1 - u_2 \|_{\CC^0_T}
    \left( 1 + \| u_1 \|_{\CC^{\alpha}_T} + \| u_2 \|_{\CC^{\alpha}_T} \right)
    .
  \end{eqnarray*}
\end{lemma}

\begin{proof}
  The bound on $\| a (u) \|_{\CC_T^{\alpha}}$ is trivial. We estimate $\| a
  (u_t) - a (u_s) \|_{\CC^0}$ as
  \[ | \int_z K_{i, x} (z) [a (u (t, z)) - a (u (s, z))] | = | \int_{z, \tau
     \in [0, 1]} K_{i, x} (z) a' (\delta_{\tau} u^{t z}_{s z}) [u (t, z) - u
     (s, z)] | \]
  \begin{eqnarray*}
    & \leqslant & | \int_{z, w, \tau \in [0, 1]} K_{i, x} (z) \sum_{j
    \lesssim k \sim i} [\Delta_j u (t, z) - \Delta_j u (s, z)] K_{k, z} (w) a'
    (\delta_{\tau} u^{t w}_{s w}) |\\
    &  & + | \int_{z, w, \tau \in [0, 1]} K_{i, x} (z) \sum_{k \lesssim j
    \sim i} [\Delta_j u (t, z) - \Delta_j u (s, z)] K_{k, z} (w) a'
    (\delta_{\tau} u^{t w}_{s w}) |\\
    &  & + | \int_{z, w, \tau \in [0, 1]} K_{i, x} (z) \sum_{k \sim j \gtrsim
    i} [\Delta_j u (t, z) - \Delta_j u (s, z)] K_{k, z} (w) a' (\delta_{\tau}
    u^{t w}_{s w}) | .
  \end{eqnarray*}
  If $k > - 1$ we have
  \[ \int_w K_{k, z} (w) a' (\delta_{\tau} u^{t w}_{s w}) = \int_w K_{k, z}
     (w) [a' (\delta_{\tau} u^{t w}_{s w}) - a' (\delta_{\tau} u^{t z}_{s z})]
  \]
  and then the first term above becomes
  \[ | \int_{z, w, \tau, \sigma} K_{i, x} (z) \sum_{j \lesssim k \sim i}
     [\Delta_j u (t, z) - \Delta_j u (s, z)] K_{k, z} (w) a'' (\delta_{\sigma}
     (\delta_{\tau} u^t_s)^w_z) [\delta u^{t w}_{t z} - \delta u^{s w}_{s z}]
     | \]
  \begin{eqnarray*}
    & \lesssim & \| a'' \|_{L^{\infty}} \| u_t - u_s \|_{\CC^0} \sum_{j
    \lesssim k \sim i} \int_w | K_{k, z} (w) | | w - z |^{\alpha} \| u
    \|_{\CC^{\alpha}_T}\\
    & \lesssim & i 2^{- \alpha i} \| a'' \|_{L^{\infty}} \| u_t - u_s
    \|_{\CC^0} \| u \|_{\CC^{\alpha}_T} .
  \end{eqnarray*}
  The second term is
  \begin{eqnarray*}
    &  & | \int_{z, w, \tau \in [0, 1]} K_{i, x} (z) \sum_{j \sim i}
    [\Delta_j u (t, \cdummy) - \Delta_j u (s, \cdummy)] (z) P_{j, z} (w) a'
    (\delta_{\tau} u^{t w}_{s w}) |\\
    & \lesssim & \| u_t - u_s \|_{\CC^0} \| P_{j, z} (w) a' (\delta_{\tau}
    u^{t w}_{s w}) \|_{L^{\infty}} \lesssim \| u_t - u_s \|_{\CC^0} \| a'
    \|_{L^{\infty}}
  \end{eqnarray*}

  The third term can be estimated as the first one when $k > - 1$. Otherwise
  we just bound it as
  \begin{eqnarray*}
    &  & | \int_{z, w, \tau \in [0, 1]} K_{i, x} (z) \sum_{j \leqslant N}
    [\Delta_j u (t, \cdummy) - \Delta_j u (s, \cdummy)] (z) K_{- 1, z} (w) a'
    (\delta_{\tau} u^{t w}_{s w}) |\\
    & \lesssim & \| u_t - u_s \|_{\CC^0} \| a' \|_{L^{\infty}} .
  \end{eqnarray*}
  For the three terms together we have the bound
  \[ \| a (u_t) - a (u_s) \|_{\CC^0} \lesssim \| u_t - u_s \|_{\CC^0} \left(
     \| a' \|_{L^{\infty}} + \| a'' \|_{L^{\infty}} \| u
     \|_{\CC^{\alpha}_{_T}} \right) \]
  With the same technique we obtain
  \begin{eqnarray*}
    \| a (u_1) - a (u_2) \|_{C_T^{\alpha / 2} \CC^0} & \lesssim & \| a'
    \|_{L^{\infty}} \| u_1 - u_2 \|_{C_T^{\alpha / 2} \CC^0}\\
    &  & + \| a'' \|_{L^{\infty}} \| u_1 - u_2 \|_{C_T^{\alpha / 2} \CC^0} \|
    u_1 - u_2 \|_{\CC^{\alpha}_T}\\
    &  & + \| a''' \|_{L^{\infty}} \| u_1 - u_2 \|_{C_T^{\alpha / 2} \CC^0}
    \| u_1 - u_2 \|^2_{\CC^{\alpha}_T}
  \end{eqnarray*}
  and this gives the second estimate. The third one can be obtained easily.
\end{proof}

\subsection{Bony's paraproduct}\label{ssec:bony}

\

In terms of Littlewood--Paley blocks, the product of two smooth functions
$fg$ can be decomposed as
\[ fg = \sum_{j \geqslant - 1} \sum_{i \geqslant - 1} \Delta_i f \Delta_j g =
   f \prec g + f \succ g + f \circ g. \]
Where the \tmtextit{paraproducts} $f \prec g$ and $f \succ g$ and the
\tmtextit{resonant product} $f \circ g$ are defined as
\[ f \prec g = g \succ f \assign \sum_{j \geqslant - 1} \sum_{i = - 1}^{j - 2}
   \Delta_i f \Delta_j g \qquad \text{and} \qquad f \circ g \assign \sum_{|i -
   j| \leqslant 1} \Delta_i f \Delta_j g. \]
We will often use the shortcuts $\sum_{i \sim j}$ for $\sum_{| i - j |
\leqslant 1}$ and $\sum_{i \lesssim j}$ for $\sum_{i < j - 1}$. Of course, the
decomposition depends on the dyadic partition of unity used to define the
blocks $\Delta_j$, and also on the particular choice of the pairs $(i, j)$ in
the diagonal part. The choice of taking all $(i, j)$ with $|i - j| \leqslant
1$ into the diagonal part corresponds to property~\ref{property-2-dyadic} in
our definition of dyadic partitions of unity.

Bony's crucial observation~{\cite{Bony1981,meyer_remarques_1981}} is that the
paraproduct $f \prec g$ (and thus $f \succ g$) is always a well-defined
distribution. Heuristically, $f \prec g$ behaves at large frequencies like $g$
(and thus retains the same regularity), and $f$ provides only a frequency
modulation of $g$. The only difficulty in constructing $fg$ for arbitrary
distributions lies in handling the resonant product $f \circ g$. The basic
result about these bilinear operations is given by the following estimates.

\begin{theorem}
  (Paraproduct estimates)\label{thm:paraproduct} For any $\beta \in
  \mathbbm{R}$ and $f, g \in \CD'$ we have
  \begin{equation}
    \label{eq:para-1} \|f \prec g\|_{\CC^{\beta}} \lesssim_{\beta}
    \|f\|_{L^{\infty}} \|g\|_{\CC^{\beta}},
  \end{equation}
  and for $\alpha < 0$ furthermore
  \begin{equation}
    \label{eq:para-2} \|f \prec g\|_{\CC^{\alpha + \beta}} \lesssim_{\alpha,
    \beta} \|f\|_{\CC^{\alpha}} \|g\|_{\CC^{\beta}} .
  \end{equation}
  For $\alpha + \beta > 0$ we have
  \begin{equation}
    \label{eq:para-3} \|f \circ g\|_{\CC^{\alpha + \beta}} \lesssim_{\alpha,
    \beta} \|f\|_{\CC^{\alpha}} \|g\|_{\CC^{\beta}} .
  \end{equation}
\end{theorem}

Bony proved also a basic \tmtextit{paralinearisation} result, soon after
improved by Meyer. We give here a particular version suited to our purposes.

\begin{theorem}
  \label{th:paralinearisation}Let $\alpha \in (0, 1)$, $f \in \left(
  \CC^{\alpha} \right)^k$ and $F \in C^3 (\mathbbm{R}^k ; \mathbbm{R})$ then
  \[ R_F (f) \assign F (f) - F' (f) \prec f \in \CC^{2 \alpha} \]
  with
  \[ \| R_F (f) \|_{\CC^{2 \alpha}} \lesssim \| F \|_{C^2} \left( 1 + \| f
     \|_{\CC^{\alpha}} \right)^2 . \]
  Moreover the map $f \mapsto R_F (f)$ is locally Lipshitz and
  \[ \| R_F (f) - R_F (\tilde{f}) \|_{\CC^{2 \alpha}} \lesssim \| F \|_{C^3}
     \left( 1 + \| f \|_{\CC^{\alpha}} + \| \tilde{f} \|_{\CC^{\alpha}}
     \right)^2 \| \tilde{f} - f \|_{\CC^{\alpha}} . \]
\end{theorem}

The additional key ingredient at the core of the paracontrolled approach is
the following commutation result proved
in~{\cite{gubinelli_paracontrolled_2015}}, Lemma~2.4:

\begin{lemma}
  \label{lem:commutator} Assume that $\alpha, \beta, \gamma \in \mathbb{R}$
  are such that $\alpha + \beta + \gamma > 0$ and $\beta + \gamma < 0$. Then
  for $f, g, h \in C^{\infty}$ the trilinear operator
  \[ C (f, g, h) \assign ((f \prec g) \circ h) - f (g \circ h) \]
  allows for the bound
  \begin{equation}
    \label{eq:commutator bound} \|C (f, g, h)\|_{\CC^{\beta + \gamma}}
    \lesssim \|f\|_{\CC^{\alpha}} \|g\|_{\CC^{\beta}} \|h\|_{\CC^{\gamma}},
  \end{equation}
  and can thus be uniquely extended to a bounded trilinear operator
  \[ C : \CC^{\alpha} \hspace{-0.17em} \times \hspace{-0.17em} \CC^{\beta}
     \hspace{-0.17em} \times \hspace{-0.17em} \CC^{\alpha} \rightarrow
     \CC^{\beta + \gamma} . \]
\end{lemma}

We will need the following two lemmas to compare standard and time-smoothed
paraproducts. The first one has essentially the same proof as
{\cite{gubinelli_paracontrolled_2015}}, Lemma~5.1.

\begin{lemma}
  \label{l:comm-paraproduct} Let $\rho \in (0, 2)$, $\gamma \in \mathbbm{R}$.
  Then for every $\varepsilon > 0$ we have the bound
  \[ \| g \prec h - g \precprec h \|_{\CC^{\rho + \gamma - \varepsilon}_T}
     \lesssim \| g \|_{C_T^{\rho / 2} \CC^0} \| h \|_{\CC^{\gamma}_T} . \]
\end{lemma}

The second lemma has a standard proof.

\begin{lemma}
  \label{l:comm-nonlin-para}Let $g \in \LL^{\rho}_T$, $h \in C_{\eta}^1
  \CC^{\gamma}_T$ with $\rho \in (0, 1)$, $\gamma \in \mathbbm{R}$. We have,
  $\forall \varepsilon > 0$
  \[ \| \Pi_{\prec} (g, h) - \Pi_{\precprec} (g, h) \|_{\CC^{\rho + \gamma -
     \varepsilon}_T} \lesssim \| g \|_{\LL^{\rho}_T} \| h \|_{C_{\eta}^1
     \CC^{\gamma}_T} . \]
\end{lemma}

\end{document}